\pdfoutput=1
\documentclass[11pt,twoside,reqno]{amsart}
\usepackage{amsfonts,amssymb,amsmath,amsthm,amsxtra,amscd}
\usepackage{mathtools}
\usepackage{bm}
\usepackage{enumitem}
\usepackage{dsfont}
\usepackage{stmaryrd}

\usepackage[top=1.5in, bottom=1.5in, left=1.23in, right=1.23in]{geometry}

\newcommand{\F}{\mathcal{F}}
\newcommand{\RR}{\mathbb R}
\newcommand{\R}{\RR}
\newcommand{\TT}{\mathbb T}
\newcommand{\NN}{{\mathbb N}}
\newcommand{\N}{\NN}
\newcommand{\Z}{\mathbb Z}
\newcommand{\QQ}{{\mathbb Q}}
\newcommand{\Q}{\QQ}
\newcommand{\C}{{\mathbb C}}

\newcommand{\PP}{\mathbb{P}}
\newcommand{\G}{\mathbb{G}}

\newcommand{\ind}[1]{\mathds{1}_{{#1}}}

\numberwithin{equation}{section}

\newtheorem{theorem}[equation]{Theorem}
\newtheorem{proposition}[equation]{Proposition}

\newtheorem{lemma}[equation]{Lemma}
\theoremstyle{definition}
\newtheorem{remark}[equation]{Remark}
\newtheorem{example}[equation]{Example}
\newtheorem{definition}[equation]{Definition}

\usepackage{hyperref} 
\hypersetup{
    breaklinks=true,
	colorlinks=true,       
	linkcolor=blue,          
	citecolor=magenta,        
	filecolor=magenta,      
	urlcolor=black           
}

\title{Remarks on the Ionescu--Wainger multiplier theorem}

\author{Dariusz Kosz}
\address{Dariusz Kosz (\textnormal{dariusz.kosz@pwr.edu.pl}) 
\newline Faculty of Pure and Applied Mathematics, Wroc{\l}aw University of Science and Technology, Wybrze{\.z}e Stanis{\l}awa Wyspia{\'n}skiego 27, 50-370 Wroc{\l}aw, Poland}

\author{Mariusz Mirek}
\address{Mariusz Mirek (\textnormal{mariusz.mirek@rutgers.edu}) 
\newline Department of Mathematics, Rutgers University, Piscataway, NJ 08854-8019, USA
\& Instytut Matematyczny, Uniwersytet Wroc{\l}awski,
Plac Grunwaldzki 2/4, 50-384 Wro\-c{\l}aw, Poland}

\author{Wojciech S{\l}omian}
\address{Wojciech S{\l}omian (\textnormal{wojciech.slomian@uga.edu})
\newline Department of Mathematics, University of Georgia, Athens, GA 30602, USA \& Instytut \linebreak Matematyczny, Uniwersytet Wroc{\l}awski, Plac Grunwaldzki 2/4, 50-384 Wro\-c{\l}aw, Poland}

\author{Jianghao Zhang}
\address{Jianghao Zhang
(\textnormal{J.Zhang-395@sms.ed.ac.uk})
\newline School of Mathematics, University of Edinburgh, Edinburgh, EH9 3FD, UK}

\begin{document}

\thanks{Dariusz Kosz was partially supported by the National Science Centre of Poland grant SONATA BIS 2022/46/E/ST1/00036. Mariusz Mirek was partially supported  by the NSF CAREER grant DMS-2236493. Wojciech S\l omian was partially supported by AMS--Simons Travel Grant No. 00008780 and by the Foundation for Polish Science scholarship START~068.2023. For the purpose of Open Access the authors have applied a~CC BY public copyright license to any Author Accepted Manuscript version arising from this submission.}

\begin{abstract}
In this paper, we extend the recent Ionescu--Wainger multiplier theorem for the set of canonical fractions by Kosz, Mirek, Peluse, Wan, and Wright in several directions. First, we prove its weighted version, which allows us to combine a multifrequency setting with appropriate arithmetic weights. Second, we establish useful seminorm variants of the theorem. Third, we improve the norm upper bounds and, surprisingly, show that these bounds cannot be uniform in the size of the family of canonical fractions. Finally, we demonstrate how these refinements (especially handling arithmetic weights) can be applied by giving a short proof of Bourgain's pointwise ergodic theorem for polynomial iterates.
\end{abstract}

\maketitle
	
\section{Introduction}
	
\subsection{A brief history}
A common phenomenon in the field of discrete analogs in harmonic analysis is that, for many important convolution operators with arithmetic features, their multipliers $\mathfrak m(\xi)$ can be efficiently approximated by certain multifrequency expressions, which locally take the following form
\begin{align}
\label{eq:38}
\mathfrak m_{\rm dis}(a/q) \cdot \mathfrak m_{\rm cont}(\xi - a/q)
\end{align}
if $\xi\in\mathbb T$ lies close to a rational point $a/q$ with a small denominator, and by $0$ otherwise. Here $\mathfrak m_{\rm dis}$ denotes the arithmetic component, typically evaluated at $a/q$ as above, while $\mathfrak m_{\rm cont}$ is the continuous component that reflects the behavior of the continuous prototype of the operator. One of the best-known examples is a discrete  averaging Radon operator given by
\begin{align}
\label{eq:36}
A_{N;\mathbb Z}^{P}f(x)\coloneqq \frac{1}{N}\sum_{n\in [N]}f(x-P(n)),
\qquad x \in \mathbb Z,
\end{align}
for $N\in\mathbb Z_+$ and a polynomial  $P\in\Z[{\rm n}]$ with integer coefficients, see Section~\ref{sec:not} for notation.

Averages \eqref{eq:36} were extensively studied by Bourgain \cite{B1, B2, B3} in the context of his celebrated pointwise ergodic theorem with polynomial iterates, see Theorem~\ref{bourgain} in Section~\ref{sec:B}. In \cite{B1, B2, B3}, Bourgain clearly emphasized that pointwise convergence problems in ergodic theory live in a very close symbiosis with quantitative estimates in the integer shift system, such as maximal or variational inequalities among others, see Sections~\ref{sec:var}~and~\ref{sec:B}. The latter inequalities ensure pointwise convergence in any measure-preserving system via the Calder{\'o}n transference principle \cite{Cald}, which nowadays is the key tool linking the world of quantitative estimates in ergodic theory (see Theorem~\ref{bourgain1}) with the world of discrete analogs in harmonic analysis, see the discussion in Section~\ref{sec:B} below formula \eqref{eq:13}.

An important fact is that the operator $A_{N;\mathbb Z}^{P}$ is given by convolution with the kernel
\begin{align*}
K_N^P\coloneqq \frac{1}{N}\sum_{n\in[N]}\ind{\{P(n)\}}.
\end{align*}
Using the Fourier transform, see Section~\ref{sec:not} for notation, we have $\mathcal F_{\mathbb Z}A_{N; \mathbb Z}^{P} f(\xi)=m_{N}(\xi)\mathcal F_{\mathbb Z}f(\xi)$ for $\xi\in\mathbb T$, where the multiplier $m_N$ is a normalized  exponential sum of the form
\begin{align}
\label{eq:37}
m_{N}(\xi) \coloneqq \mathcal F_{\mathbb Z}K_N^P(\xi)= \frac{1}{N}\sum_{n\in[N]}e(-\xi\cdot P(n)).
\end{align}

Bourgain \cite{B1, B2, B3} had an ingenious insight to use the circle method of Hardy--Littlewood--Ramanujan to analyze the multipliers from \eqref{eq:37} and to make the approximation \eqref{eq:38} rigorous. Namely, for $N\ge 1$, we define a continuous counterpart of $m_N$ by
\begin{align*}
\mathfrak{m}_N(\xi) \coloneqq \int_{0}^1e(-\xi\cdot P(Nt)) \,{\rm d}t, \qquad \xi\in\mathbb R.
\end{align*}
Also, for every  $a/q\in\mathbb Q$, we  define the complete exponential sum by  $G(a/q) \coloneqq m_q(a/q)$. As a  consequence of the mean-value theorem, for every $\xi\in\mathbb T$ and $(a,q) \in\mathbb Z \times \mathbb Z_+$ satisfying ${\rm gcd}(a, q)=1$, if $|\xi-a/q|\le N^{-\deg P+\delta}$ and $q\le N^\delta$, then one has
\begin{align}
\label{eq:40}
m_N(\xi)=G(a/q)\mathfrak m_N(\xi-a/q)+E_N(\xi),
\end{align}
where $\|E_N\|_{L^\infty(\mathbb T)}\le N^{-1/2}$ provided that $\delta\in(0, 1)$ is sufficiently small.

Identity \eqref{eq:40} is the key mechanism for constructing approximations on the major arcs in the classical circle method, thereby justifying \eqref{eq:38} whenever $\xi\in\mathbb T$ is in a small vicinity of some $a/q\in\mathbb Q$ such that $q \in \Z_+$ is small, i.e. $|\xi-a/q|\le N^{-\deg P+\delta}$ and $q\le N^\delta$. This was first observed by Bourgain in the context of discrete operators of Radon type \eqref{eq:36} and their connections to pointwise ergodic theorems.

In order to gain a better perspective on the identity \eqref{eq:40} and its role in the circle method, we refer to a vast literature in analytic number theory on the classical circle method; the books by Iwaniec and Kowalski \cite{IK}, Nathanson \cite{Nat}, Vaughan \cite{Vau}, Vinogradov \cite{Vin}, and the notes of Wooley \cite{Wooleynotes} are excellent references on this exciting subject.

Identity \eqref{eq:40}, together with the classical Weyl inequality \cite{Weyl} for the exponential sums from \eqref{eq:37}, led Bourgain \cite{B1, B2, B3} to construct an approximation for \eqref{eq:37}, namely
\[
m_N(\xi)=\mathfrak a_N(\xi)+ E_N^*(\xi),
\]
where the error term $E_N^*$ satisfies $\|E_N^*\|_{L^\infty(\mathbb T)}\le N^{-\gamma}$ for some small $\gamma\in(0, 1)$ provided that $\delta\in(0, 1)$ is sufficiently small, and $\mathfrak a_N$ is a multifrequency multiplier given by the formula
\begin{align}
\label{eq:41}
\mathfrak a_N(\xi)\coloneqq\sum_{a/q\in \mathcal R^1_{\le N^\delta}}G(a/q)\mathfrak m_N(\xi-a/q)\eta\big(N^{\deg P-\delta}(\xi-a/q)\big),
\end{align}
where $\eta \colon \R\to[0, 1]$ is a smooth   bump function so that $\ind{[-1/4, 1/4]}\le \eta\le \ind{(-1/2, 1/2)}$, and
\begin{align*}
\mathcal R_{\le N}^1 \coloneqq \big\{ a/q \in \QQ / \Z : q \in \{1,\dots,N\}
\text{ and } {\rm gcd}(a, q)=1 \big \}
\end{align*}
is the set of \textit{canonical fractions}, where $a/q$ is identified with its equivalence class $a/q + \Z$.

The key outcome of this approximation is that the multiplier \(\mathfrak a_N\) is supported on the union of small disjoint neighborhoods of canonical fractions \(a/q \in \mathcal R_{\le N^\delta}^1\) provided that $\delta\in(0, 1)$ is sufficiently small. Thus, it has precisely the form of the expression in \eqref{eq:38}. By the disjointness of the summands in \eqref{eq:41} and Plancherel's theorem, the operator $f \mapsto \mathcal F_{\mathbb{Z}}^{-1} (\mathfrak a_N \mathcal F_{\mathbb{Z}} f)$ is a contraction on \(\ell^2(\mathbb{Z})\) even though the cardinality of \(\mathcal R_{\le N^\delta}^1\) grows to infinity with \(N\). In fact, the latter is the main obstruction to extending the theory beyond \(\ell^2(\mathbb{Z})\).

After that reduction, the \(\ell^p(\mathbb{Z})\) bounds of maximal functions associated with \eqref{eq:36}, or more generally other seminorms arising in the study of pointwise convergence problems, have been reduced to understanding the corresponding objects involving the operators \(f\mapsto \mathcal F_{\mathbb{Z}}^{-1}(\mathfrak a_N \mathcal F_{\mathbb{Z}} f)\) instead of $A_{N; \mathbb Z}^{P} f$. However, the multifrequency nature of the multiplier \(\mathfrak a_N\) (i.e. \(\mathfrak a_N\) is a sum of  pieces with disjoint supports centered at frequencies $a/q\in\mathcal R_{\le N^\delta}^1$) significantly complicates matters, even though each object, such as
\[
G(a/q)\quad \text{ or } \quad \mathfrak m_N(\xi-a/q)\eta\big(N^{\deg P -\delta}(\xi-a/q)\big),
\]
is well understood from the perspective of number theory and harmonic analysis for each individual $a/q\in\mathcal R_{\le N^\delta}^1$. These complications are immediately apparent at the level of the $\ell^p(\Z)$ bounds for the operator \(f\mapsto \mathcal F_{\mathbb{Z}}^{-1}(\mathfrak a_N \mathcal F_{\mathbb{Z}} f)\) when $p\neq 2$, even for each $N\in\Z_+$ separately, since Plancherel's theorem is not available on $\ell^p(\mathbb{Z})$ spaces with $p \neq 2$.

Bourgain instead developed a deep machinery in \cite{B1,B2,B3} to study discrete operators with multifrequency multipliers \eqref{eq:41}. However, this approach required separating the arithmetic components $G(a/q)$ from the continuous part $\mathfrak m_N(\xi-a/q)\eta\big(N^{\deg P-\delta}(\xi-a/q)\big)$ by employing intricate methods from harmonic analysis and number theory, among others.

As far as we know, all available methods in the field of discrete analogs in harmonic analysis necessitate separate analysis of the arithmetic part and the continuous part. Our aim in this paper is to demonstrate that the discrete operators corresponding to multipliers like in \eqref{eq:41} can be handled in a fairly unified way, without parceling out the arithmetic components from the continuous part, which will significantly simplify proof strategies.

\subsection{Statement of the main results} More specifically, our aim is to prove a variant of the Ionescu--Wainger multiplier theorem \cite{IW} that incorporates arithmetic weights. Loosely speaking, this variant will allow us to work with discrete operators corresponding to multifrequency multipliers as in \eqref{eq:41} and to analyze the arithmetic and continuous parts simultaneously. For this purpose, we introduce the new concept of Ionescu--Wainger weights. These weights will have certain algebraic properties, which ensure that the Ionescu--Wainger theory remains adaptable in the corresponding weighted setting. Here and throughout, we will use the basic notation introduced in Section~\ref{sec:not}.
    
\begin{definition}[Ionescu--Wainger weights]\label{def:IW-weight}
A nonzero function  $G \colon \QQ^d\to\mathbb{C}$ is an \emph{Ionescu--Wainger weight} if it satisfies the following two conditions.

\begin{enumerate} [label=(\Alph*), itemsep=2pt]
    \item \label{cond:A} Factorization modulo $1$: if $a/q \equiv a_1/q_1 + a_2/q_2 \pmod{1}$ with ${\rm gcd}(q_1,q_2)=1$, then
\[
G(a/q) = G(a_1/q_1) \cdot G(a_2/q_2).
\]
In particular, $G$ is $1$-periodic and $G(0)=1$.
    \item \label{cond:B} Uniform Fourier transform boundedness: there exists a constant $U_G \in \R_+$ such that
\[
\sup_{q\in\Z_+}\| \mathcal F^{-1}_{(\Z / q \Z)^d} (G) \|_{L^1({(\Z / q \Z)^d})} \leq U_G.
\]
\end{enumerate}
For an Ionescu--Wainger weight $G$ and a finite set $\Sigma \subset \QQ^d$, we let
\begin{align}
\label{eq:42}
{\bf C}_{G}(\Sigma) \coloneqq U_G \| G \|_{\ell^\infty(\Sigma)}.
\end{align}
\end{definition}

\begin{example}
\label{ex:1}
The following functions are typical Ionescu--Wainger weights:

\begin{enumerate} [label=(\alph*), itemsep=2pt]
\item \label{itema} The constant function $G \equiv 1$. Note that in this case we have $U_{G} = 1$.
\item \label{itemb} The exponential function $a/q \mapsto e(a/q)$.
\item \label{itemc} Gauss-type sums $a/q \mapsto q^{-k} \sum_{n \in [q]^k} e(P(n) \cdot a/q)$ for any polynomial mapping $P=(P_1,\ldots, P_d) \colon \Z^k\to \Z^d$, where the components $P_1,\ldots, P_d$ are $k$-variate polynomials with integer coefficients.
\end{enumerate}
It is not difficult to see that these functions satisfy condition~\ref{cond:A}. In Section~\ref{sec:not}, we collect further remarks on the Ionescu--Wainger weights. In particular, we prove that each of the examples above satisfies also condition~\ref{cond:B}, see inequality \eqref{eq:Gauss-IW}.
\end{example}

For every Ionescu--Wainger weight $G$, since $G=\mathcal F_{(\Z / q \Z)^d}\mathcal F^{-1}_{(\Z / q \Z)^d} (G)$,  we have
\[
\| G \|_{L^\infty(q^{-1}\Z^d)} \leq \| \mathcal F^{-1}_{(\Z / q \Z)^d} (G) \|_{L^1({(\Z / q \Z)^d})} \leq U_G.
\]
In particular, $1 = G(0) \leq \sup_{\beta \in \QQ^d} |G(\beta)| \leq U_G$ so that ${\bf C}_{G}(\Sigma)\le U_G^2$ by \eqref{eq:42}. 

Using the notation from \eqref{IWeq:372} and \eqref{eq:1}, we are now ready to state the first main result of this article, which concerns a multiplier as in \eqref{eq:38} with an arithmetic component $G$ and a continuous component $\mathfrak m$, localized to the set of canonical fractions $\mathcal R^d_{\le N}$. 

\begin{theorem} \label{thm:IW:upper} Fix $r \in \Z_+$. Then there exists a constant ${\bf C}_{\rm IW}(r) \in \R_+$ such that, for each $d \in \Z_+$ and each $N \in \N_{\geq 10}$, the following is true. Assume that
\begin{align} \label{eq:support}
0 < \varepsilon \leq (4r N^{2r})^{-1}
\end{align}
and let $\mathfrak m \colon \RR^{d} \to L(H_1,H_2)$ be a measurable function supported on $\varepsilon \mathbf{Q}^d$, whose values are bounded linear operators between two separable Hilbert spaces $H_{1}$ and $H_{2}$. Let $\mathbf{A}_{2r} \in [0,\infty]$ be such that, for all $f \in L^2(\RR^d; H_1) \cap L^{2r}(\RR^d; H_1)$, we have
\begin{equation*}
\|T_{\RR^d}[\mathfrak m] f\|_{L^{2r}(\RR^d; H_2)} \leq \mathbf{A}_{2r} \|f\|_{L^{2r}(\RR^d; H_1)}.\end{equation*}
Suppose that $G \colon \TT^d \to \mathbb{C}$ is an Ionescu--Wainger weight from Definition~\ref{def:IW-weight}. Then
\begin{equation}\label{IW UPPER}
\big\| T_{\Z^d} \big[ \llbracket G \mid \mathfrak m\rrbracket_{\mathcal R^d_{\le N}}\big] f \big\|_{\ell^{p}(\Z^d; H_2)} \lesssim_{d,r}
N^{{\bf C}_{\rm IW}(r) / \log \log N}
\mathbf{C}_{G}(\mathcal R^d_{\leq N})
\mathbf{A}_{2r} \|f\|_{\ell^{p}(\Z^d; H_1)}
\end{equation}
for all $f \in \ell^2(\Z^d; H_1) \cap \ell^{p}(\Z^d; H_1)$ with $(2r)/(2r-1) \leq p \leq 2r$, and $\mathbf{C}_{G}(\mathcal R^d_{\leq N})$ as in \eqref{eq:42}.
\end{theorem}

Theorem~\ref{thm:IW:upper} originates in the groundbreaking paper of Ionescu and Wainger \cite{IW}. The authors in \cite{IW} developed a deep multifrequency multiplier theorem in the spirit of Theorem~\ref{thm:IW:upper} that allowed them (which was out of reach prior to \cite{IW}) to establish $\ell^p(\Z^d)$ boundedness for all $p\in(1, \infty)$ of the discrete singular Radon transforms
\begin{align}
\label{eq:45}
R_N^Pf(x) \coloneqq \sum_{m\in \{-N, \dots, N\}^k\setminus\{0\}^k }f(x-P(m))K(m), \qquad x\in\Z^d,
\end{align}
where $N\in\Z_+$ and $P \colon \Z^k\to\Z^d$ is a polynomial mapping with integer coefficients, see item \ref{itemc} in Example~\ref{ex:1}, and $K$ is a Calder{\'o}n--Zygmund kernel, see \cite[Theorem~1.1]{IW}.   

Ionescu and Wainger \cite{IW}, assuming that for any arbitrary $\rho\in(0, 1)$ the condition
\begin{align}
\label{eq:44}
0<\varepsilon < e^{-N^{\rho}}
\end{align}
is satisfied in place of condition \eqref{eq:support}, originally proved Theorem~\ref{thm:IW:upper} in the scalar setting, i.e. for $H_1=H_2=\C$, with $G\equiv1$ and a bound $(\log N)^{2/\rho+1}$ in place of $N^{{\bf C}_{\rm IW}(r) / \log \log N} $ in \eqref{IW UPPER}, and for the so-called Ionescu--Wainger fractions in place of the canonical fractions $\mathcal R^d_{\leq N}$ defined in \eqref{IWeq:372}. The Ionescu--Wainger fractions are defined as follows
\begin{align}
\label{eq:43}
\widetilde{\mathcal R}^d_{\leq N}\coloneqq\big\{ a/q \in(\QQ/ \Z)^d: q \in P_N \text{
and } {\rm gcd} (a, q)=1\big\},
\end{align}
where $P_N$ is a certain large family of positive integers having delicate factorization properties that dictated condition \eqref{eq:44}, which is much stronger than condition \eqref{eq:support}.

Bourgain's results \cite{B1, B2, B3} and the Ionescu--Wainger theorem \cite{IW} share the same starting point --- the circle method --- but otherwise have completely different proofs.  Although \cite{IW} was a huge breakthrough, the full strength of the multifrequency multiplier theorem was not yet fully understood. The situation changed when the second author \cite{M1} highlighted that the Ionescu--Wainger multifrequency multiplier theorem can be interpreted as a multifrequency Littlewood--Paley theory that captures all the Diophantine features arising from the Hardy--Littlewood circle method. This gave a positive answer to Stein's question about the existence of a discrete Littlewood--Paley theory and was used to obtain a new proof of the $\ell^p(\Z^d)$ boundedness for all $p\in(1, \infty)$ of the discrete singular Radon transforms \eqref{eq:45} from \cite{IW}. The ideas from \cite{M1} led to many papers on the $r$-variational \cite{MST1, MST2}, jump \cite{MSZ3}, or uniform oscillation \cite{MSS} inequalities for discrete operators of Radon type.

Currently, the Ionescu--Wainger multifrequency multiplier theorem \cite{IW} is a key tool in the field of discrete analogs in harmonic analysis, and it has undergone a period of considerable changes and developments. As mentioned above, it was originally proved for the so-called Ionescu--Wainger fractions \eqref{eq:43}. In \cite{M1}, the second author, in addition to observing that the Ionescu--Wainger multiplier theorem can be interpreted as a multifrequency Littlewood--Paley theory, also proved that the norm is at most of order $\log N$. In \cite{MSZ3}, the authors extended the Ionescu--Wainger multiplier theorem to the setting of arbitrary separable Hilbert spaces $H_1$ and $H_2$. This was valuable for applications to various square functions. In an important paper \cite{Pierce}, Pierce undertook a successful effort to abstract the key orthogonality mechanisms behind the proof of the Ionescu--Wainger multiplier theorem in terms of the superorthogonality phenomena in harmonic analysis. This resulted in a very clean proof of the Ionescu--Wainger multiplier theorem. Essentially at the same time, Tao \cite{TaoIW} provided an alternative proof of the Ionescu--Wainger multiplier theorem with bounds independent of $N$. We emphasize that all these improvements to the Ionescu--Wainger multiplier theorem \cite{M1, MSZ3, Pierce, TaoIW} were achieved under the condition \eqref{eq:44} for the set of Ionescu--Wainger fractions $\widetilde{\mathcal R}^d_{\leq N}$ from \eqref{eq:43}, as in the original Ionescu--Wainger approach \cite{IW}.

Recently, the first two authors in collaboration with Peluse, Wan, and Wright \cite{KMPWW}, answered the question posed by Ionescu and Wainger in \cite[Remark~3, p.~361]{IW} in the affirmative, proving a variant of the Ionescu--Wainger multiplier theorem, as stated in Theorem~\ref{thm:IW:upper} with $G\equiv 1$ and the norm $N^{{\bf C}_{\rm IW}(r) \log \log \log N / \log \log N}$ in place of $N^{{\bf C}_{\rm IW}(r) / \log \log N}$ in \eqref{IW UPPER}. This new Ionescu--Wainger multiplier theorem from \cite{KMPWW} was essential in establishing a multilinear variant of Weyl's inequality, which was key to Bergelson's conjecture \cite{KMPWW} for multilinear ergodic averages with commuting transformations for iterates along polynomials with distinct degrees. Theorem~\ref{thm:IW:upper} can also be used to interpret the classical Weyl inequality for exponential sums \cite{Weyl} in $\ell^p(\mathbb{Z})$ spaces for $p \in (1, \infty)$. The details can be found in \cite{M-ICM}. 

Now, having this brief outline of recent advances in the Ionescu--Wainger multiplier theorem, we can highlight the novelties in our results.

\begin{enumerate}[label*={(\arabic*)}, itemsep=2pt]
\item As far as we know, this is the first instance in the literature where the formulation of the Ionescu--Wainger multiplier theorem incorporates the arithmetic weights \(G\) satisfying conditions \ref{cond:A} and \ref{cond:B}. Surprisingly, the proof of the Ionescu--Wainger multiplier theorem with \(G \equiv 1\) from \cite{KMPWW} is easily adaptable to our weighted setting without many differences. In Section~\ref{sec:logloglog}, for the convenience of the reader, we present all the details, even those that were omitted in \cite{KMPWW}. Thus, the price we pay for having a weighted version of the Ionescu--Wainger multiplier theorem, see Theorem~\ref{thm:IW:upper} above, is not high, but the benefits from having such a result are substantial. Now the multipliers $\llbracket G \mid \mathfrak m\rrbracket_{\mathcal R^d_{\le N}}$ defined in \eqref{eq:1} can be handled in a fairly unified way without parceling out the arithmetic weights \(G\) from the continuous part \(\mathfrak m\).

\item In Section~\ref{sec:var}, we prove Theorem~\ref{thm:IWvar}, which is a seminorm variant of Theorem~\ref{thm:IW:upper}, and will handle all seminorms arising in \cite{MSS, MST1, MST2, MSZ3}, see Definition~\ref{def:RM}. In Section~\ref{sec:B}, we will show how our weighted variants of the Ionescu--Wainger theory lead to unified and  short proofs of all results from \cite{MSS, MST1, MST2, MSZ3}, see Theorem~\ref{bourgain1}.

\item These weighted formulations of Theorem~\ref{thm:IW:upper} and Theorem~\ref{thm:IWvar}, together with the fact that they work with the set of canonical fractions \(\mathcal R^d_{\le N}\) from \eqref{IWeq:372}, are among the most important outcomes of this paper that significantly simplify proof strategies. The first outcome avoids using complicated approximations handling the multipliers involving arithmetic components. The second avoids misalignment arising from working with the Ionescu--Wainger fractions \(\widetilde{\mathcal R}^d_{\le N}\) from \eqref{eq:43}.

\item  Here we also note that the norm in \eqref{IW UPPER} is at most of order \(N^{\mathbf{C}_{\rm IW}(r)/\log\log N}\), which improves the upper bound \(N^{\mathbf{C}_{\rm IW}(r)\log\log\log N / \log\log N}\) from \cite{KMPWW}. This is achieved by a slightly more delicate treatment of the divisor functions that arise in our argument. In fact, our bound \(N^{\mathbf{C}_{\rm IW}(r)/\log\log N}\) coincides with the upper bound for the divisor function. In Theorem~\ref{thm:IW:upper}, it is important that the norm in \eqref{IW UPPER} is controlled by \(N^{\gamma}\) for any fixed \(\gamma\in(0,1)\), which covers all practical applications of inequality \eqref{IW UPPER} that we are aware of. Of course the latter bound is implied by our upper bound.

\item A weighted analog of Theorem~\ref{thm:IW:upper} for the Ionescu--Wainger fractions \(\widetilde{\mathcal R}^d_{\le N}\) in place of canonical fractions \(\mathcal R^d_{\le N}\) can be obtained by adapting the arguments from Tao's paper \cite{TaoIW}. As in \cite{TaoIW}, the implied bound in \eqref{IW UPPER} is independent of \(N\). 

\item Surprisingly, the dependence of the norm in inequality \eqref{IW UPPER} for \(\mathcal R^d_{\le N}\) on the parameter \(N\) cannot be removed, which stands in sharp contrast to Tao's result \cite{TaoIW}. This provides a negative answer to \cite[Problem~1.43]{KMPWW}, which was rather unexpected.

\item Although the bounds in \eqref{IW UPPER} for \(\widetilde{\mathcal R}^d_{\le N}\) are uniform in \(N\), in contrast to our non-uniform bounds in \(N\) for \(\mathcal R^d_{\le N}\) given in \eqref{IW UPPER}, the advantage of using Theorem \(\ref{thm:IW:upper}\) is twofold. First, it aligns with the set of canonical fractions \(\mathcal R^d_{\le N}\) that naturally arise in the circle method. Second, the set of canonical fractions allows us to relax the decay in condition \eqref{eq:44} from exponential to polynomial, as in \eqref{eq:support}. The latter was essential in establishing a multilinear variant of Weyl's inequality in \cite{KMPWW}.
\end{enumerate}

The penultimate remark can be quantified, leading to the second main result of the paper.
       
\begin{theorem} \label{thm:IW:lower} Let $p \in (1,\infty) \setminus \{2\}$. Then there exists $c_p \in (0,1)$ such that the following is true. Fix a nonzero Schwartz function $\mathfrak m \colon \R \to \C$ supported on ${\bf Q}$ and let $(\vartheta_N)_{N \in \Z_+}$ be a~given sequence of numbers $\vartheta_N \in (0,1)$. Then, for infinitely many $N \in \Z_+$, we have
\begin{equation*}
\bigl\| T_{\Z} \big[ \llbracket 1 \mid  {\mathfrak m}_{\varepsilon_N}\rrbracket_{\mathcal R_{\le N}}\big]  \bigr\|_{\ell^{p}(\Z) \to \ell^{p}(\Z)}
\gtrsim_{p}
e^{(\log N)^{c_p}} = N^{1 / (\log N)^{1-c_p}}
\end{equation*}
with some $0 < \varepsilon_N \leq \vartheta_N$ and ${\mathfrak m}_{\varepsilon_N}(x) \coloneqq {\mathfrak m}(x/\varepsilon_N)$.
\end{theorem} 

Theorem~\ref{thm:IW:lower} demonstrates that applications of inequality \eqref{IW UPPER} that would require its norm to be independent of the parameter \(N\) will not be possible. This is only possible with the Ionescu--Wainger fractions \(\widetilde{\mathcal R}^d_{\le N}\), as was proved by Tao \cite{TaoIW}. 

Finally, let us emphasize that the method of proof of Theorem~\ref{thm:IW:lower} can be adapted to show that condition \ref{cond:B} imposed on our weights in Definition \ref{def:IW-weight} is necessary; if dropped, one can show that Theorem~\ref{thm:IW:upper} fails, see Remark~\ref{rem:0} in Section~\ref{sec:lower}.

\subsection{Organization of the paper}  In Section~\ref{sec:not}, we gather the necessary notation and some remarks on the weights from Definition \ref{def:IW-weight}. In Section~\ref{sec:logloglog} and Section~\ref{sec:lower}, we prove Theorem~\ref{thm:IW:upper} and Theorem~\ref{thm:IW:lower}, respectively. In Section~\ref{sec:var}, we introduce the concept of Rademacher--Menshov seminorms (see Definition~\ref{def:RM}) and establish a seminorm variant of Theorem~\ref{thm:IW:upper} that generalizes \cite[Theorem~3.32]{KMPWW}. Finally, in Section~\ref{sec:B}, we illustrate how our weighted Ionescu--Wainger theory works in practice, giving a relatively short and conceptually clean proof of Bourgain's pointwise ergodic theorem for polynomial iterates, see Theorem~\ref{bourgain}. To this end, we prove its quantitative versions in Theorem~\ref{bourgain1}. The latter provides the sharpest quantitative information on pointwise convergence for \eqref{eq:36}.

\section{Notation and simple tools}
\label{sec:not}
We now set up the notation that will be used throughout the article. 

\subsection{Basic notation} 
Let $\Z_+ \coloneqq \{1,2,\dots\}$ be the set of positive integers, $\NN \coloneqq \{0,1,2,\dots\}$ be the set of nonnegative integers, and $\RR_+ \coloneqq (0,\infty)$ be the set of positive real numbers. Given $d \in \Z_+$, the sets $\Z^d, \, \Q^d, \, \RR^d$, and $ \C^d$ have their standard meaning. The $d$-dimensional unit cube in $\RR^d$ centered at the origin is denoted by ${\bf Q}^d \coloneqq [-1/2,1/2]^d$.

If $A\subseteq\C^d$, then its translations and dilations are defined by $A+u\coloneqq\{a+u: a\in A\}$ for all $u\in \C^d$ and $vA\coloneqq\{va: a\in A\}$ for all $v\in\C$, respectively.

For $N \in \R_+$, we define $[N] \coloneqq \Z \cap (0,N]$. For $N \in \R$  and  $A\subseteq \R$, we also define the sets
\begin{align*}
A_{<N}&\coloneqq A\cap(-\infty, N),  \quad \phantom{],} A_{>N}\coloneqq A\cap(N, \infty),\\
A_{\le N}&\coloneqq A\cap(-\infty, N],  \quad 
\phantom{),} A_{\ge N}\coloneqq A\cap[N, \infty).
\end{align*}

Here and throughout, the $d$-dimensional torus $\TT^d \coloneqq \RR^d / \Z^d$ is endowed with the $\ell^\infty$ metric on $\TT^d$ induced by the norm $| \xi |_\infty \coloneqq \sup_{i \in [d]} {\rm dist}(\xi_i, \Z)$. Moreover, we will identify $\TT^d$ with $[0,1)^d$ or, up to a set of measure zero, with ${\bf Q}^d$, as will be clear from the context.

Finally, if $A$ is a set, then its indicator function is denoted by $\ind{A}$. When $A$ is finite, its number of elements is denoted by $|A|$. If $S$ is a statement, then $\ind{S}$ denotes its indicator, equal to $1$ if $S$ is true and $0$ if $S$ is false. In particular, we have $\ind{x\in A} \equiv \ind{A}(x)$.

\subsection{Asymptotic notation} Throughout the paper, $C\in\RR_+$ denotes an absolute constant whose value may change from line to line. For two quantities $A,B\in\RR_{\ge 0}$, we write $A \lesssim B$ or $B \gtrsim A$ if $A\le CB$ for some $C\in\RR_+$. If $A \lesssim B\lesssim A$, then we write $A \simeq B$. We will use the symbols $\lesssim_{\delta}$ or $\simeq_{\delta}$ to emphasize that the implicit constant $C$ depends on a parameter $\delta$.
    
\subsection{Function spaces}
All vector spaces in this paper will be defined over the complex numbers $\C$. The triple $(X, \mathcal B(X), \mu)$ is a measure space $X$ with $\sigma$-algebra $\mathcal B(X)$ and $\sigma$-finite measure $\mu$. The space of all measurable functions  whose modulus is integrable with $p$-th power is denoted by $L^p(X, \mu)$ for $p\in(0, \infty)$, whereas $L^{\infty}(X, \mu)$ denotes the space of all essentially bounded measurable functions. More generally, if $(B, \|\, \cdot \, \|_B)$ is a  separable normed vector space, then  $L^{p}(X,\mu;B)$ denotes the space of all measurable functions $F \colon X\to B$ (up to $\mu$-almost everywhere equivalence) such that
\begin{align*}
\|F\|_{L^{p}(X, \mu;B)} \coloneqq \left\|\|F\|_B\right\|_{L^{p}(X, \mu)}<\infty.
\end{align*}

We shall abbreviate $L^p(X, \mu)$ to $L^p(X)$ and $L^{p}(X,\mu;B)$ to $L^{p}(X;B)$. If $X$ is endowed with counting measure, then we further abbreviate $L^p(X)$ to $\ell^p(X)$ and $L^p(X; B)$ to $\ell^p(X; B)$.

For a continuous linear map $T \colon B_1 \to B_2$ between two normed vector spaces $B_1$ and  $B_2$,  its operator norm will be denoted by $\|T\|_{B_1 \to B_2}$. The space of all continuous linear mappings $T \colon B_1 \to B_2$ will be denoted by $L(B_1,B_2)$. 

\subsection{Fractions} 
For integers $p, q\in\Z_+$, we write $p\mid q$ if $p$ divides $q$ or, equivalently, $q \in p \Z$.

For $a = (a_1,\ldots, a_d) \in \Z^d$ and $q\in\Z_+$, we denote by ${\rm gcd}(a,q)$ the largest positive integer $n\in\Z_+$ that divides each of the integers $q,a_1, \ldots, a_d$. Observe that any element of $\QQ^d$ has a unique representation as $a/q$ with $a \in \Z^d$ and $q\in \Z_{+}$ such that ${\rm gcd}(a,q)=1$.

Given $N \in \Z_+$, define the $1$-periodic sets of canonical fractions by
\begin{align} \label{IWeq:372} \mathcal R_{\le N}^d \coloneqq \big\{ a/q \in(\QQ/ \Z)^d: q \in [N] \text{
and } {\rm gcd} (a, q)=1\big\}.
\end{align}
Here each element of ${\mathcal R}^d_{\le N}$ is represented by a real rational $a/q \in \QQ^d$ which we identify with the equivalence class $a/q + \Z^d$. Similarly, for $q \in \Z_+$, we define the set of $q$-quotients by $\mathcal Q(q) \coloneqq (q^{-1} \Z / \Z)^d$ and the set of reduced fractions by $\mathcal R(1) \coloneqq \mathcal Q(1)$ or, when $q>1$, by
\[
{\mathcal R}(q) \coloneqq \{ a/q \in \mathcal Q(q) : {\rm gcd}(a,q) = 1 \}.
\] 

\subsection{Convolutions} Let $(\G, +)$ be a locally compact abelian group (LCA group) equipped with a Haar measure $\lambda_{\G}$. For two functions $f, g\in L^1(\G)$ we define their convolution by 
\begin{align*}
f*_{\G}g(x) \coloneqq f*g(x) \coloneqq\int_{\G}f(x-y)g(y)\,{\rm d}\lambda_{\G}(y).
\end{align*}
For the sake of simplicity, we will always abbreviate $*_{\G}$ to $*$; this notation will always be clear from the context and will not cause any confusion. Since $\G$ is abelian, we readily see that $f*g=g*f$. We will be mainly working with $\G=\R^d$ or $\G=\TT^d$  endowed with Lebesgue measure or $\G=\Z^d$ endowed with counting measure.

\subsection{Fourier transform}
We shall write $e(z)=e^{2\pi {\bm i} z}$ for every $z\in\C$, where ${\bm i}^2=-1$.  As before let $(\G, +)$ be an LCA group equipped with a Haar measure $\lambda_{\G}$. It is well known (see for instance \cite{Rudin}) that every LCA group $\G$ has a Pontryagin dual $\hat \G = (\hat \G,+)$, an LCA group with a Haar measure $\lambda_{\hat \G}$ and a pairing, i.e.~a continuous bihomomorphism $\G \times \hat \G\ni (x,\xi) \mapsto \langle x , \xi\rangle\in\TT$, such that the Fourier transform $\F_{\G} \colon L^1(\G) \to L^\infty(\hat \G)$ given by
\[
 \F_{\G} f(\xi) \coloneqq \int_{\G} f(x) e(-\langle x , \xi\rangle)\,{\rm d}\lambda_{\G}(x), \qquad \xi\in\hat \G,
\]
extends to a unitary map from $L^2(\G)$ to $L^2(\hat \G)$; in particular, we have Plancherel's identity
\[
\|\F_\G f\|_{L^2(\hat \G)}=\|f\|_{L^2(\G)}, \qquad f\in L^2(\G).
\]
Moreover, the inverse Fourier transform $\F_\G^{-1} \colon L^2(\hat \G) \to L^2(\G)$ is given by the formula
\[
 \F_\G^{-1} f(x) = \int_{\hat \G} f(\xi) e(\langle x , \xi\rangle)\,{\rm d}\lambda_{\hat \G}(\xi), \qquad x\in\G,
\]
whenever $f\in L^1(\hat \G) \cap L^2(\hat \G)$.

We will mainly work with concrete pairs $(\G,\hat \G)$ of Pontryagin dual LCA groups:
\begin{itemize}
\item [(i)]  If $\G = \R^d$ with Lebesgue measure $\lambda_{\R^d} = {\rm d}x$, then $\hat \G = \R^d$ with Lebesgue measure $\lambda_{\R^d} = {\rm d}\xi$ is a Pontryagin dual, with pairing $\langle x , \xi\rangle \coloneqq x \cdot \xi$. Also, for any $f \in L^1(\R^d)$,
\begin{align*}
\mathcal F_{\R^d} f(\xi)  \coloneqq  \int_{\R^d} f(x) e(-x\cdot\xi) \,{\rm d}x, \qquad \xi\in\R^d.
\end{align*}
\item[(ii)]  If $\G = \Z^d$ with counting measure $\lambda_{\Z^d}$, then $\hat \G = \TT^d$ with Lebesgue measure $\lambda_{\TT^d} = {\rm d}\xi$ is a Pontryagin dual, with pairing $\langle x , \xi\rangle \coloneqq x \cdot \xi$. Also, for any $f \in \ell^1(\Z^d)$,
\begin{align*}
\mathcal F_{\Z^d}f(\xi) \coloneqq \sum_{x \in \Z^d} f(x) e(-x\cdot\xi), \qquad \xi\in \TT^d.
\end{align*}
\item[(iii)] If $\G = (\Z/q\Z)^d$ for some $q \in \Z_+$ with normalized counting measure $\lambda_{(\Z/q\Z)^d}$, then $\hat{\G} = (q^{-1}\Z/\Z)^d$ with counting measure $\lambda_{(q^{-1}\Z/\Z)^d}$ is a Pontryagin dual, with pairing $\langle x , \xi\rangle \coloneqq x \cdot \xi$. Also, for any $f \in L^1((\Z/q\Z)^d)$,
\begin{align*}
\mathcal F_{(\Z/q\Z)^d}f(\xi) \coloneqq q^{-d}\sum_{x \in (\Z/q\Z)^d} f(x) e(-x\cdot\xi), \qquad \xi\in (q^{-1}\Z/\Z)^d.
\end{align*}
\end{itemize}

\subsection{Multiplier operators}  
For $\mathfrak m \in L^\infty(\hat{\G})$, we define the multiplier operator by  
\begin{align*}
T_{\G}[\mathfrak m]f(x) \coloneqq \int_{\hat{\G}} \mathfrak m(\xi) \, \mathcal F_{\G}f(\xi) \, e(x \cdot \xi) \, {\rm d}\lambda_{\hat{\G}}(\xi), \qquad x\in \G,
\end{align*}
where $f \colon \G \to \C$ is any function for which this formula makes sense.

\subsection{Multifrequency extensions} 
Let $\mathfrak m \in L^\infty(\hat{\G})$, a finite set $\Sigma \subseteq \hat{\G}$, and $G \colon \hat{\G}\to \mathbb C$ be given. We define a \textit{weighted multifrequency extension} of $\mathfrak m$ with respect to $\Sigma$ by setting
\begin{align}
\label{eq:1}
\big\llbracket G \mid  \mathfrak m\big\rrbracket_{\Sigma}(\xi)\coloneqq\sum_{\sigma\in\Sigma}G(\sigma)\mathfrak m(\xi-\sigma), \qquad \xi\in \hat{\G}.
\end{align}
If $G\equiv 1$, then $\llbracket 1 \mid  \mathfrak m\rrbracket_{\Sigma}$ is the \textit{multifrequency extension} of $\mathfrak m$ with respect to  $\Sigma$. If $\Sigma=\{\sigma\}$, then we write $\llbracket G \mid  \mathfrak m\rrbracket_{\Sigma}=\llbracket G \mid  \mathfrak m\rrbracket_{\sigma}$. If $\Sigma$ is the singleton of the neutral element of $\hat{\G}$, then we simply have $\llbracket 1 \mid  \mathfrak m\rrbracket_{\Sigma}=\mathfrak m$.

\subsection{Gaussian weights}
Let $w\colon \mathbb{Z}^k \to \mathbb{C}$ be a weight that satisfies for some $C_w \in \R_+$ the following $\ell^1$ bound
\begin{equation}\label{eq:w-L1}
\sup_{q\in\Z_+}q^{-k}  \sum_{r\in[q]^k} |w(r)| \le C_w.
\end{equation}
Here, we shall show that, for any function $\Phi \colon \Z^k\to \Z^d$, a \textit{generalized Gaussian sum}
\begin{align*}
G_w(a/q)\coloneqq q^{-k}\sum_{r\in[q]^k}w(r) \,
e(\Phi(r) \cdot a/q), \qquad a/q\in(\Q/\Z)^d,
\end{align*}
satisfies condition~\ref{cond:B}. Indeed, fix $d, k \in \Z_+$ and $\Phi\colon \mathbb{Z}^k\to\mathbb{Z}^d$. We show that $G_w$ satisfies condition~\ref{cond:B} with $U_G = C_w$, that is,
\begin{equation}\label{eq:Gauss-IW}
   \sup_{q \in \Z_+} q^{-d}
    \sum_{m \in (\Z / q \Z)^d}
    \Big|
    \sum_{a \in [q]^d} 
    G_w(a/q)\, e(m \cdot a/q)
    \Big|
    \le C_w.
\end{equation}

Fix $q \in \Z_+$. For each $m\in{(\Z / q \Z)^d}$, we have  
\[
\mathcal F^{-1}_{(\Z / q \Z)^d} (G_w)(m)
= \sum_{a \in [q]^d} G_w(a/q)\, e(m \cdot a/q).
\]
Substituting the definition of $G_w$, interchanging sums, and using orthogonality, we obtain
\begin{align*}
    \mathcal F^{-1}_{(\Z / q \Z)^d} (G_w)(m) = 
    q^{-k} \sum_{r\in[q]^k} w(r)
       \sum_{a \in [q]^d}
       e\big( (\Phi(r) + m) \cdot a / q \big)=q^{-k+d}
   \sum_{\substack{r\in[q]^k\\ \Phi(r) + m\in q\Z^d}} w(r).
\end{align*}
By taking absolute values, summing over $m$, and interchanging summation order, we have
\begin{align*} 
  \sum_{m \in (\Z / q \Z)^d} |\mathcal F^{-1}_{(\Z / q \Z)^d} (G_w)(m)|
  \leq 
  q^{-k+d}
  \sum_{r\in[q]^k} |w(r)|,
\end{align*}
since for each $r\in[q]^k$ exactly one $m \in (\Z / q \Z)^d$ satisfies $\Phi(r) + m \in q\Z^d$. Then \eqref{eq:w-L1} gives
\[
\| \mathcal F^{-1}_{(\Z / q \Z)^d} (G_w) \|_{L^1({(\Z / q \Z)^d})} =
  q^{-d} \sum_{m \in (\Z / q \Z)^d} |\mathcal F^{-1}_{(\Z / q \Z)^d} (G_w)(m)|
  \le 
  q^{-k} \sum_{r\in[q]^k} |w(r)|
  \leq
  C_w,
\]
as desired. In particular, if $w$ is nonnegative and normalized by $\sum_{r\in[q]^k} w(r) = q^k$, then $\mathcal F^{-1}_{(\Z / q \Z)^d} (G_w) \ge 0$ and the above inequality becomes an equality with $C_w = 1$.

\section{Proof of Theorem~\ref{thm:IW:upper}: the upper bound} \label{sec:logloglog}

In this section, we prove Theorem~\ref{thm:IW:upper}. Although its content shares many similarities with \cite[Section~2]{MSZ3} or \cite[Section~3]{KMPWW}, there are some subtle differences that may cause difficulties, so for convenience we include detailed explanations to ease the exposition. We begin with the formulation of the sampling principle of Magyar, Stein, and Wainger \cite{MSW}.

\begin{theorem}
\label{prop:msw}
For each $d\in\Z_+$, there exists ${\bf C}_{\rm MSW}(d) \in \R_+$ such that the following holds. Let $p \in [1,\infty]$ and $q\in\Z_+$, and let $B_1$ and $B_2$ be two finite-dimensional Banach spaces. If $\mathfrak m \colon \RR^d \to L(B_1, B_2)$ is a measurable function supported on $q^{-1}\mathbf{Q}^d$, then
\[
\big\|T_{\Z^d} \big[ \llbracket 1 \mid  \mathfrak m\rrbracket_{{\mathcal Q(q)}}\big]\big\|_{\ell^{p}(\Z^d;B_1)\to \ell^{p}(\Z^d;B_2)}\le {\bf C}_{\rm MSW}(d) \|T_{\RR^d}[\mathfrak m]\|_{L^{p}(\RR^d;B_1)\to L^{p}(\RR^d;B_2)}.
\]
\end{theorem}
\begin{proof}
The proof of Theorem~\ref{prop:msw} can be found in \cite[Corollary~2.1, p.~196]{MSW}. We also refer to \cite{MSZ1} for a generalization of Theorem~\ref{prop:msw} to real interpolation spaces.
\end{proof}

\subsection{Simple reductions} 
We begin the proof of Theorem~\ref{thm:IW:upper} with a list of reductions.

\begin{enumerate}[label=(\roman*), itemsep=2pt]
\item Since the case $r=1$ follows by Plancherel's theorem, we may assume that $r \in \N_{\geq 2}$.
\item By duality and interpolation, it suffices to prove \eqref{IW UPPER} for $p = 2r$.
\item By replacing $\mathfrak m$ with $\tilde{\mathfrak m} \colon \RR^d \to L(H_1\oplus H_2, H_1\oplus H_2)$ defined by 
\[
\tilde{\mathfrak m}(x) (h_1, h_2) \coloneqq (0,\mathfrak m(x)(h_1)),
\]
we may assume that $H_1=H_2=H$.
\item By standard limiting arguments, we may assume that $H$ is finite-dimensional.
\item We may assume that $\mathbf{A}_{2r} \in (0,\infty)$ and then normalize $\mathfrak m$ so that $\mathbf{A}_{2r} = 1$.
\item Since Theorem~\ref{prop:msw} applied with $q=1$ gives \eqref{IW UPPER} with $N^{C_r / \log \log N} \mathbf{C}_{G}(\mathcal R^d_{\leq N})$ replaced by $N^{d+1}$, we may assume that $N \geq N_0$ for a fixed large integer $N_0\in\Z_+$. 
\end{enumerate}

After all these reductions, our aim is to prove \eqref{IW UPPER} with $p=2r$, which reads as follows 
\begin{equation*}
\big\| T_{\Z^d} \big[ \llbracket G \mid  \mathfrak m\rrbracket_{\mathcal R^d_{\le N}}\big] f \big\|_{\ell^{2r}(\Z^d; H)} \lesssim_{d,r}
N^{{\bf C}_{\rm IW}(r) / \log \log N}
\mathbf{C}_{G}(\mathcal R^d_{\leq N}) \|f\|_{\ell^{2r}(\Z^d; H)}.
\end{equation*}

\subsection{Uniqueness property}
We recall the following definition, which was essential in \cite{IW}.

\begin{definition}[Uniqueness property]
\label{up}
A sequence $(v_i)_{i \in [n]}$ has the \emph{uniqueness property} if there exists $j\in[n]$ such that $v_j\neq v_i$ for all $i\in[n]\setminus\{j\}$. In other words, the element $v_j$ occurs in the sequence $(v_i)_{i \in [n]}$ exactly once. This definition will primarily be used for sequences of rational numbers, but it is also applicable to more general sequences.
\end{definition}

Using Definition~\ref{up}, we now formulate an important abstract orthogonality principle.

\begin{proposition}[Abstract orthogonality principle]
\label{UP}
Given $r$ and $H$ as before, there exists ${\bf C}_{\rm UP}(r) \in \RR_+$ such that the following is true. Let $(X,\mathcal B(X), \mu)$ be a measure space and let $V_{1},\dots,V_{r} \subset \Z_+$ be finite sets, not necessarily disjoint. For every pair $(i,v)$ with $v \in V_{i}$, fix
$F_{i,v} \in L^{2r}(X;H)$. Suppose that, for each $2r$-tuple
\begin{equation*}
(v_{1,1}, v_{1,2}, v_{2,1}, v_{2,2}, \ldots, v_{r,1}, v_{r,2})
\in V_{1}^{2} \times V_{2}^{2} \times \dotsm \times V_{r}^{2}
\end{equation*}
with the uniqueness property, we have
\begin{align}
\label{IWeq:6}
\int_{X} \prod_{i \in [r]} \langle F_{i,v_{i,1}}(x), F_{i,v_{i,2}}(x) \rangle_{H} \, {\rm d}\mu(x) = 0. 
\end{align}
Then
\begin{equation*}
\int_{X} \prod_{i \in [r]} \Big \| \sum_{v \in V_{i}} F_{i,v}(x) \Big\|_{H}^{2} \, {\rm d}\mu(x)
\leq {\bf C}_{\rm UP}(r)
\int_{X} \prod_{i \in [r]} \Big( \sum_{v \in V_{i}} \| F_{i,v}(x)\|_{H}^{2} \Big) \, {\rm d}\mu(x).
\end{equation*} 
\end{proposition}

\begin{proof}
For a detailed proof, see \cite[Corollary~2.24]{MSZ3}. We also refer to \cite[Lemma~2.2, p.~363]{IW} for a scalar version of this principle.
\end{proof} 

Our next goal is to partition the set of canonical fractions \eqref{IWeq:372} into a controlled number of pieces to which Proposition~\ref{UP} can be applied efficiently. Here our approach significantly deviates from \cite{IW} and \cite{MSZ3}.  In these two papers an essential role was played by the Ionescu--Wainger rationals whose denominators were factored into two parts, one being a relatively large product of powers of small prime divisors, and the other being a product of a small number of powers of large prime divisors. In contrast to \cite{IW} and \cite{MSZ3}, instead of working with the Ionescu--Wainger rationals, we work directly with the set of canonical fractions \eqref{IWeq:372} and partition it using the new concept of $N$-lifted composites introduced in \cite[Section~3]{KMPWW}.

\begin{definition}[$N$-lifted composites]
\label{N-lifted}
Let $\mathbb P$ be the set of all prime numbers. For $N\in\Z_{+}$, an \emph{$N$-lifted composite} $Q_q$ of an integer $q \in [N]$ is defined by setting  
\begin{align*}
Q_q \coloneqq \prod_{p \in \PP \, : \, p | q} p^{\lfloor \log_p N \rfloor}
\end{align*}
with the convention $Q_1 \coloneqq 1$. In other words, $q$ and $Q_q$ are divisible by the same prime numbers, while $Q_q$ is a product of integers of the form $p^{\lfloor \log_p N \rfloor}$. Obviously, we have $q | Q_q$.
\end{definition}

We recall the concept of property $\mathcal O_N$ from \cite{KMPWW}, a variant of property $\mathcal O$ from \cite{IW}. 

\begin{definition}[Property $\mathcal O_N$]
\label{propO}
Fix $N \in \N_{\geq N_0}$ and define $K_N \coloneqq \lfloor \frac{5 \log N}{\log \log N} \rfloor$. We say that a~finite set $\Lambda \subset \Z_+$ has \emph{property $\mathcal O_N$} if $\Lambda \subseteq S_{1}\dotsm S_{K}$ for some $K \in [K_N]$, where 
\[
S_{1}\dotsm S_{K} \coloneqq \{s_{1}\dotsm s_{K} : s_1\in S_{1},\ldots, s_K \in S_{K}\}
\]
with $S_1, \dots, S_K$ being disjoint subsets of the set $\{ p^{\lfloor \log_p N \rfloor} : p \in \mathbb P \cap \N_{\leq N} \}$.
\end{definition}

Let us enumerate the set of all prime numbers $\mathbb P = \{p_m : m \in \Z_+ \}$ in such a way that $p_{m'} > p_{m}$ holds if $m' > m$. The following lemma explains the restriction $K \in [K_N]$. 

\begin{lemma} \label{L1}
    If $N \in \Z_+$ is sufficiently large, then $p_1 p_2 \cdots p_{K_N+1} > N$. 
\end{lemma}

\begin{proof}
	Given $n \in \Z_+$, its primorial $n \#$ is the product of all $p \in \PP  \cap \N_{\leq n}$. We know that 
    \[
    \lim_{n \to \infty} \sqrt[n]{n \#} = e.
    \]
    Thus, if $n$ is sufficiently large, then $n \# > 2^n$ and, by the prime number theorem, $n \#$ has no more than $\lfloor \frac{2n}{\log n} \rfloor$ factors. Hence, the product of the first $\lfloor \frac{4n}{\log n} \rfloor$ prime numbers exceeds $4^n$. Now, assuming that $N$ is sufficiently large, we take $n \coloneqq \lfloor \log N \rfloor + 1$ and obtain
    \[
    4^n > e^{\log N} = N
    \quad \text{and} \quad
    K_N + 1 \geq 5 \log N / \log \log N \geq \lfloor 4n / \log n \rfloor.
    \]
    Therefore, we have $p_1 p_2 \cdots p_{K_N+1} \geq p_1 p_2 \cdots p_{\lfloor 4n / \log n \rfloor} > 4^n > N$, as desired. 
\end{proof}

From now on, we assume that $N_0\in\Z_+$ is so large that, in particular, Lemma~\ref{L1} and the upper bound $|\mathbb P \cap \N_{\leq N}| \leq \lfloor \frac{2 N}{\log N} \rfloor$ from the prime number theorem are applicable if $N \geq N_0$.

\subsection{Partitioning of the set of canonical fractions}
For $N\in \N_{\geq N_0}$, let $I_N \coloneqq \lfloor \frac{2 N}{\log N} \rfloor$. As observed before, for every $a/q \in \mathcal R^d_{\leq N}$, the corresponding composite $Q_q$ consists of at most $K_N$ factors of the form $p^{\lfloor \log_p N \rfloor}$, where each $p$ is one of the first $I_N$ prime numbers. This motivates the following counting lemma that originates in \cite{M1}.

\begin{lemma} \label{L2} There exists ${\bf C}_{\rm sur} \in \RR_+$ such that the following is true. For $N \in \N_{\geq N_0}$, set $J_N \coloneqq \lfloor {\bf C}_{\rm sur} 2^{{\bf C}_{\rm sur} \log N / \log \log N} \rfloor$ and fix $I \in [I_N]$ and $K \in [K_N]$ with $K \leq I$. Let $V$ be a set of size $I$. Then, for some $J \in [J_N]$, there exist surjective functions $ g_{1},\dotsc,g_{J} \colon V \to [K]$ such that, whenever $E \subseteq V$ is of size $K$, there exists $j \in [J]$ for which $g_{j}(E)$ is also of size $K$.
\end{lemma}

\begin{proof}
Observe that, denoting $J \coloneqq J_N$, it suffices to find arbitrary (not necessarily distinct and not necessarily surjective) functions $g_{1},\dotsc,g_{J} \colon V \to [K]$ with the desired property.

Set $\binom{V}{K} \coloneqq \{E \subseteq V : |E| = K\}$ and let $K^{V}$ consist of all functions $g \colon V\to [K]$. Then
\[
| \{g\in K^{V} : |g(E)|=K \} | = \alpha_K |K^{V}|
= \alpha_K K^I
\]
holds for each $E \in\binom{V}{K}$ with $\alpha_K \coloneqq K! / K^K$. Suppose for a contradiction that 
\[
\big \{(g_1,\dotsc, g_J)\in (K^{V})^J : \text{for all
} E \in\tbinom{V}{K} \text{ there is } j \in [J] \text{ such that
} |g_j(E)|=K \big \}=\emptyset,
\]
that is, for each $J$-tuple, there is $E \in \binom{V}{K}$ such that $|g_j(E)| < K$ for all $j \in [J]$. Then
\[
K^{I J} \leq \sum_{E \in\binom{V}{K}} \big| \big\{ (g_1,\dotsc, g_J) \in (K^{V})^J : |g_j(E)| < K 
\text{ for all } j \in [J] \big\} \big| 
\leq \big| \tbinom{V}{K} \big| (1 - \alpha_K)^J K^{I J}
\]
and so $1 \leq |\binom{V}{K}| ( 1 - \alpha_K )^J$. Note that $I \leq I_N = \lfloor \frac{2 N}{\log N} \rfloor$ and $K \leq K_N = \lfloor \frac{5 \log N}{\log \log N} \rfloor$. Hence,
\[
1 \leq \big| \tbinom{V}{K} \big| ( 1 - \alpha_K )^J < I^K e^{-\alpha_K J} = e^{K \log I - J K! / K^K} \leq e^{K \log I - J e^{-K}} < 1
\]
holds by $J = J_N = \lfloor {\bf C}_{\rm sur} 2^{{\bf C}_{\rm sur} \log N / \log \log N} \rfloor$, provided that ${\bf C}_{\rm sur}$ is sufficiently large.
\end{proof}

Proceeding in two steps, we partition $\mathcal R^d_{\le N}$ into pieces, each having property $\mathcal O_N$.
 
\medskip \paragraph{\bf Step~1}
By Lemma~\ref{L1}, we partition $\mathcal R^d_{\le N}$ into disjoint subsets $ \mathcal R^d_{N,0} \coloneqq \mathcal R(1)$ and $\mathcal R^d_{N,K}$ for  $K \in [K_N]$, consisting of all fractions $a/q \in \mathcal R^d_{\le N}$ whose denominators $q$ have exactly $K$ prime divisors. The multiplier operator corresponding to $ \mathcal R^d_{N,0}$ can be readily handled by using Theorem~\ref{prop:msw}. Hence, from now on, we focus only on $\mathcal R^d_{N,K}$ with $K \in [K_N]$. 

\medskip \paragraph{\bf Step~2} Fix $K \in [K_N]$. We apply Lemma~\ref{L2} with 
\[
V \coloneqq \{ p^{\lfloor \log_p N \rfloor} : p \in \mathbb P \cap \N_{\leq N} \} = \{ Q_{p} : p \in \mathbb P \cap \N_{\leq N} \},
\]
which produces surjective functions $g_1,\ldots, g_J \colon V\to [K]$  with $J \in [J_N]$. For every $j \in[J]$ and $k \in [K]$, we define $S_k^j \coloneqq g_j^{-1}(\{k\}) = \{ v \in V : g_j(v) = k \}$. Next, we iteratively define
\begin{align*}
 \mathcal O(g_1) & \coloneqq S_1^1\cdots S_K^1, \\
 \mathcal O(g_2) & \coloneqq S_1^2\cdots S_K^2 \setminus \mathcal O(g_1), \\
 & \ \ \vdots \\
 \mathcal O(g_J) & \coloneqq S_1^J\cdots S_K^J \setminus \big(\mathcal O(g_1) \cup \dots \cup \mathcal O(g_{J-1})\big).
\end{align*}
Finally, for each $j \in [J]$, we set
\begin{align}
\label{IWeq:8}
\mathcal R^d_{N,K,j} \coloneqq \big\{ a / q \in \mathcal R^d_{N,K}: Q_q\in \mathcal O(g_j) \big\}.
\end{align}
The sets $\mathcal R^d_{N,K,j}$ are disjoint and the corresponding sets $\mathcal O(g_j)$ have property $\mathcal O_N$. Moreover, $\mathcal R^d_{N,K}= \mathcal R^d_{N,K,1} \cup \dots \cup \mathcal R^d_{N,K,J}$. Indeed, if $a/q \in \mathcal R^d_{N,K}$, then $q$ has exactly $K$ prime divisors, say $p_1, \ldots, p_K\in\mathbb  P$, and $Q_q=Q_{p_1}\cdots Q_{p_K}$. By Lemma~\ref{L2}, there exists $j\in[J]$ such that $|g_j(\{Q_{p_1}, \dots, Q_{p_K}\})|=K$. Taking the smallest index $j$ with this property,  we obtain $Q_q=Q_{p_1}\cdots Q_{p_K}\in \mathcal O(g_j)$ and consequently $a/q \in \mathcal R^d_{N,K, j}$, as desired.

The product structure in $S_1^j\cdots S_K^j$ is crucial for employing denominator and numerator orthogonality arguments. An alternative way to reveal such a product structure is to apply the random partition technique appearing in the proof of \cite[Theorem~2.3]{TaoIW}. 

To complete the proof, it suffices to prove inequality \eqref{IW UPPER} with $p=2r$ and the multiplier $\llbracket G \mid \mathfrak m\rrbracket_{\mathcal R^d_{\le N}}$ replaced by $\llbracket G \mid \mathfrak m\rrbracket_{\mathcal R^d_{N,K,j}} $ for every $K \in [K_N]$ and $j\in [J]$, with $\mathcal R^d_{N,K,j}$ defined in \eqref{IWeq:8}. For this purpose, additional notation and terminology need to be introduced.

\subsection{\textbf{\texorpdfstring{$\mathbb P$}{TEXT}-irreducibility and factorization}}

Let us first recall the concept of \emph{$\mathbb P$-irreducible fractions} from \cite{KMPWW}, which played a critical role in the verification of condition \eqref{IWeq:6}. 

\begin{definition}[Quotients, reduced fractions, and $\mathbb P$-irreducible fractions]
Fix $q\in \Z_+$. Identifying $a/q$ with $a/q + \Z^d$ for each $a \in \N_{<q}^d$, we define the following subsets of $(\Q/\Z)^d$.
\begin{enumerate}[label*={\arabic*}.]
	\setlength\itemsep{0.25em}
\item The set of \emph{quotients} or \emph{$q$-quotients} is defined by $\mathcal Q(q) \coloneqq {q}^{-1} \N_{<q}^d$.
\item The set of \emph{reduced fractions} is defined by $\mathcal R(1) \coloneqq \mathcal Q(1)$ or, when $q>1$, by 
\[
{\mathcal R}(q) \coloneqq \{ a/q\in \mathcal Q(q) : {\rm gcd}(a, q) = 1 \}.
\]
\item The set of \emph{$\mathbb P$-irreducible fractions} is defined by $\mathcal I(1) \coloneqq \mathcal Q(1)$ or, when $q>1$, by
\[
{\mathcal I}(q) \coloneqq \{ a/q \in \mathcal Q(q) : q/ {\rm gcd}(a, q) \text{ has exactly the same prime divisors as } q \}.
\]
\end{enumerate}
\end{definition}

We note that $\mathcal I(q)\supseteq \mathcal R(q)$. Moreover, $\mathcal I(q)$ can be expressed as a disjoint union of $\mathcal R(\tilde q)$, where $\tilde q$ runs over all positive divisors of $q$ having exactly the same prime divisors as $q$.

\begin{example}
For $d=1$ and $q=12$, identifying $a/q$ with $a/q + \Z$, we obtain
\begin{align*}
\quad 
\mathcal R(12)  = \bigg \{ \frac{1}{12},
\frac{5}{12}, \frac{7}{12}, \frac{11}{12} \bigg \} \quad {\rm and} \quad
\mathcal I(12)  = \bigg \{ \frac{1}{12},
\frac{2}{12}, \frac{5}{12}, \frac{7}{12}, \frac{10}{12}, \frac{11}{12} \bigg \}.
\end{align*}
\end{example}

For a finite set $\Lambda \subset \Z_+$, define 
\[
\mathcal R(\Lambda) \coloneqq \bigcup_{q \in \Lambda} \mathcal R(q)
\quad \text{and}
\quad 
\mathcal I(\Lambda) \coloneqq \bigcup_{q \in \Lambda} \mathcal I(q).
\]
Obviously, $\mathcal R(\Lambda)\subseteq \mathcal I(\Lambda)$ and the set of canonical fractions \eqref{IWeq:372} can be written as $\mathcal R([N])$ using this definition. An important feature of $\mathcal R(\Lambda)$, which was used in \cite{IW} and \cite{MSZ3}, is the following factorization property. If ${\rm gcd}(q, q') = 1$ for every pair $(q,q') \in \Lambda \times \Lambda'$, then
\[
\mathcal R(\Lambda \Lambda') = \mathcal R(\Lambda) \oplus \mathcal R(\Lambda').
\]
The direct sum $\oplus$ indicates that each element $u$ in the sumset ${\mathcal R}(\Lambda) + {\mathcal R}(\Lambda')$ has a unique representation $u = a/q + a'/q' \pmod{\Z^d}$. This follows from the Chinese remainder theorem. 

The following definition allows us to formulate the factorization property for $\mathcal I(\Lambda)$.

\begin{definition}[$\mathbb P$-separated sets]
A set of positive integers is \emph{$\mathbb P$-separated} if, for any two distinct elements, there exists $p \in \PP$ that divides exactly one of them.
\end{definition}

\begin{lemma}
\label{lem:fact}
Suppose that $\Lambda, \Lambda'\subseteq \Z_+$ are two $\mathbb P$-separated sets. If ${\rm gcd}(q, q') = 1$ for every pair $(q,q') \in \Lambda \times \Lambda'$, then the map $\Phi \colon \mathcal I(\Lambda) \times \mathcal I(\Lambda')\to \mathcal I(\Lambda \Lambda')$ given by
\[
\mathcal I(\Lambda) \times \mathcal I(\Lambda') \ni
(a/q, a'/q') \mapsto a/q + a'/q' \in
\mathcal I(\Lambda \Lambda')
\]
is a bijection. In particular, we have
\[
\mathcal I(\Lambda \Lambda') = \mathcal I(\Lambda) \oplus \mathcal I(\Lambda').
\] 
\end{lemma}

\begin{proof}
We first note that the denominator of $a/q + a'/q'=(q'a+qa')/(qq')$ in the reduced form has the same prime divisors as $qq'\in \Lambda \Lambda'$ due to the fact that ${\rm gcd}(q, q') = 1$. Hence, the map $\Phi$ indeed takes values in the set $\mathcal I(\Lambda \Lambda')$. 

Next, we show that $\Phi$ is surjective. Fix $b/u \in \mathcal I(\Lambda \Lambda')$, where $u=qq'$ for some $(q,q') \in \Lambda \times \Lambda'$. By definition, $b/u= \tilde b/\tilde u$, where ${\rm gcd}(\tilde b, \tilde u)=1$ and $\{u,\tilde u\}$ is not $\PP$-separated. By ${\rm gcd}(q, q') = 1$, we may write $\tilde u=\tilde q\tilde q'$, where both $\{q,\tilde q\}, \{q', \tilde q'\}$ are not $\PP$-separated. By the Chinese remainder theorem, we have $\tilde b / \tilde u = \tilde a/\tilde q + \tilde a'/\tilde q'$, where ${\rm gcd}({\tilde a}, {\tilde q}) = {\rm gcd}({\tilde a'}, {\tilde q'}) = 1$. Writing $\tilde a/\tilde q = a/q \in \mathcal I(\Lambda)$ and $\tilde a'/\tilde q' = a'/q' \in \mathcal I(\Lambda')$, we obtain $b/u = \Phi(a/q, a'/q')$.

Finally, we show that $\Phi$ is injective. Suppose that 
\[
a_1/q_1 + a'_1/q'_1 \equiv a_2/q_2 + a'_2/q'_2 \pmod{\Z^d}
\]
with $q_1,q_2 \in \Lambda$ and $q'_1,q'_2 \in \Lambda'$. Then ${\rm gcd}(q_1, q'_1) ={\rm gcd}(q_2, q'_2)= 1$. Rewriting the above as
\[
\tilde a_1/ \tilde q_1 + \tilde a'_1/\tilde q'_1 \equiv \tilde a_2/ \tilde q_2 + \tilde a'_2/ \tilde q'_2 \pmod{\Z^d}
\]
with ${\rm gcd}(\tilde a_1, \tilde q_1) ={\rm gcd}(\tilde a'_1, \tilde q'_1)={\rm gcd}(\tilde a_2, \tilde q_2) ={\rm gcd}(\tilde a'_2, \tilde q'_2)=1$, we deduce that $\tilde q_1 = \tilde q_2$ and $\tilde q'_1=\tilde q'_2$. This implies $q_1=q_2$ and $q'_1=q'_2$, since both $\Lambda, \Lambda'$ are $\mathbb P$-separated. Therefore, $a_1 - a_2 \in q_1 \Z^d$ and $a'_1 - a'_2 \in q'_1 \Z^d$, which means that $\Phi$ is injective, as desired. 
\end{proof}

	Fix $K \in [K_N]$ and $j\in [J]$. Referring to \eqref{IWeq:8}, we observe that
	\begin{equation*}
	\llbracket G \mid \mathfrak m\rrbracket_{\mathcal R^d_{N,K,j}}
	= 
	\llbracket G^\circ \mid \mathfrak m\rrbracket_{\mathcal I(S_1 \cdots S_K)}
	\end{equation*}
	with $S_k \coloneqq g_j[\{k\}]^{-1}$ for all $k \in [K]$ and $G^\circ \coloneqq G \cdot \ind{N,K,j}$, where $\ind{N,K,j}(a/q) \coloneqq 1$ if $a/q$ is $\mathbb P$-irreducible and upon reduction belongs to $\mathcal R^d_{N,K,j}$, and $\ind{N,K,j}(a/q) \coloneqq 0$ otherwise. The sets $S_k$ form a partition of $V$ determined by the surjection $g_j$. Moreover, the ${\mathbb P}$-irreducible fractions ${\mathcal I}(q)$ are pairwise disjoint sets as $q$ varies over the product set $S_1 \cdots S_K$. 
	
	\subsection{Denominators} We use Lemma~\ref{lem:fact} repeatedly to exploit orthogonality between the denominators. Denote $f^\circ_{\mathcal I} \coloneqq T_{\Z^d} \big[\llbracket G^\circ \mid \mathfrak m\rrbracket_{\mathcal I}\big]f$ for every $\mathcal I \subseteq \mathcal I(S_1 \cdots S_K)$. We have
	\begin{align} \label{denom-0}
	\big\| T_{\Z^d} \big[ 	\llbracket G \mid \mathfrak m\rrbracket_{\mathcal R^d_{N,K,j}} \big] f \big\|_{\ell^{2r}(\Z^d; H)}^{2r}
	= \sum_{x \in \Z^d} 
	  \big\| f^\circ_{\mathcal I(S_1 \cdots S_K)} (x) \big\|_{H}^{2r}.
	\end{align}
	
	For $L \subseteq [K]$, set ${\mathcal S}_L \coloneqq \prod_{l \in L} S_l$ with the convention ${\mathcal S}_{\emptyset} \coloneqq \{1\}$. Moreover, for $q \in {\mathcal S}_L$, we let $F_{q, L} \coloneqq f^\circ_{\mathcal I(q) \oplus \mathcal I({\mathcal S}_{[K] \setminus L})}$. Now, for each $k \in [K]$, we shall show the following inequality
	\begin{align} \label{denom}
	\sum_{x \in \Z^d} \Big(
	\sum_{q \in {\mathcal S}_{[k-1]}} 
	\big \| F_{q, [k-1]} (x) \big \|_{H}^{2} \Big)^r
	\leq C_r
	\sum_{x \in \Z^d} \Big(
	\sum_{q \in {\mathcal S}_{[k]}} 
	\big \| F_{q, [k]}(x) \big \|_{H}^{2} \Big)^r
	\end{align}
	with $C_r = {\bf C}_{\rm UP}(r)$ using Proposition~\ref{UP}. After $K$ iterations of \eqref{denom}, we shall obtain 
	\begin{align*}
	\sum_{x \in \Z^d} \Big( 
	\sum_{q \in {\mathcal S}_{\emptyset}}
	\big \| F_{q,\emptyset}(x) \big \|_{H}^{2} \Big)^r
	\leq C_r^{K}
	\sum_{x \in \Z^d} \Big( 
	\sum_{q \in {\mathcal S}_{[K]}}
	\big \| F_{q, [K]}(x) \big \|_{H}^{2} \Big)^r
	\end{align*}
	or, equivalently, 
	\begin{align} \label{denom-II}
	\sum_{x \in \Z^d} 
	\big\| f^\circ_{\mathcal I(S_1 \cdots S_K)} (x) \big\|_{H}^{2r}
	\le C_r^{K} \sum_{x \in \Z^d} 
	\Big(\sum_{q \in S_{1} \cdots S_{K}}
	\big\| f^\circ_{\mathcal I(q)} (x) \big\|_{H}^{2}\Big)^r.
	\end{align}
	The constant $C_r^{K/2r}$ is acceptable because $K \leq K_N = \lfloor \frac{5 \log N}{\log \log N} \rfloor$.

Turning to the proof of \eqref{denom}, we rewrite the left-hand side as
\[
\sum_{q_1, \dots, q_r \in {\mathcal S}_{[k-1]}}
\sum_{x \in \Z^d} 
\prod_{i \in [r]}
\Big \|
\sum_{q^* \in S_{k}} 
F_{q_i q^*, [k]}(x) \Big\|_{H}^{2}.
\] 
Fix $q_1, \dots, q_r \in {\mathcal S}_{[k-1]}$. We claim that, for every $2r$-tuple $(q^*_{1,1}, q^*_{1,2}, \dots, q^*_{r,1}, q^*_{r,2}) \in S_{k}^{2r}$ with the uniqueness property, the Fourier transform of the function
\[
\prod_{i \in [r]} \langle F_{q_i q^*_{i,1}, [k]}, F_{q_i q^*_{i,2}, [k]} \rangle_{H}
\] 
vanishes in a neighborhood of the origin so that \eqref{IWeq:6} holds. Indeed, since $H$ is separable, the Fourier transform of this function is a finite sum of functions, each supported on a set
\[
2r \varepsilon {\bf Q}^d + (v_{1,1} - v_{1,2}) + \cdots + (v_{r,1} - v_{r,2}),
\]
where $v_{i,1}, v_{i,2} \in \mathcal I(S_{[K]})$ are fractions with denominators  divisible by $q^*_{i,1}, q^*_{i,2}$ respectively for each $i\in[r]$.  Moreover,  these denominators, in reduced form, are not greater than $N$. We rewrite the above set as $2r \varepsilon {\bf Q}^d + v$ and observe that $v \notin \Z^d$ because, due to the uniqueness property, property $\mathcal O_N$, and $\mathbb P$-irreducibility, at least one of the factors $q^*_{1,1}, q^*_{1,2}, \dots, q^*_{r,1}, q^*_{r,2}$ cannot be completely eliminated from the denominator of $v$. Thus, for some $i \in [d]$, the distance between $\Z$ and the $i$-th coordinate of $v$ is at least $N^{-2r}$. By $\varepsilon \leq (4r N^{2r})^{-1}$, this implies that $\Z^d \cap (2r \varepsilon {\bf Q}^d +v) = \emptyset$. Since condition \eqref{IWeq:6} is satisfied, an application of Proposition~\ref{UP} completes the proof of inequality \eqref{denom}.

\subsection{Numerators}
Next, we exploit orthogonality between the numerators to bound the right-hand side of \eqref{denom-II}. As before, we use Lemma~\ref{lem:fact} repeatedly.  Our strategy is to split the sums in question into diagonal and off-diagonal terms by
using the following result.
\begin{lemma}
\label{L3}
Fix $r \in \Z_+$. There exists ${\bf C}_{\rm diag}(r) \in \RR_+$ such that, for every $N \in \Z_+$ and for all positive numbers $a_1, \dotsc, a_N \in \RR_+$, the following numerical inequality holds
\begin{equation*}
(a_1+\dotsb+a_N)^r
\leq {\bf C}_{\rm diag}(r) \, \Big( \sum_{n \in [N]} a_n^r+
\sum_{\substack{n_1,\dotsc,n_r \in [N]\\{\rm distinct}}} a_{n_1}\dotsm a_{n_r} \Big).
\end{equation*}
\end{lemma}
\begin{proof}
This was proved in \cite[Lemma~2.35]{MSZ3} with ${\bf C}_{\rm diag}(r) \coloneqq \max\{(r(r-1))^{r-1},2\}$.
\end{proof}

For any partition $L \cup L' \cup L'' = [K]$, we define
\[
{\mathcal F}_{L,L',L''} \coloneqq \sum_{q \in {\mathcal S}_L} \Big( \sum_{q' \in {\mathcal S}_{L'}} \sum_{q'' \in {\mathcal S}_{L''}} \sum_{u'' \in \mathcal I(q'')} \big \| 
f^\circ_{\mathcal I(q) \oplus \mathcal I(q')+u''}
\big \|_{H}^{2} \Big)^r.
\]
Then the sum on the right-hand side of \eqref{denom-II} can be rewritten as follows
\begin{align}\label{denom-III}
\sum_{x \in \Z^d} 
\Big(\sum_{q \in S_{1} \cdots S_{K}}
\big\| f^\circ_{\mathcal I(q)} (x) \big\|_{H}^{2}\Bigr)^r
= \| {\mathcal F}_{\emptyset, [K], \emptyset} \|_{\ell^1(\Z^d)}.
\end{align}
Assume that $l \in L'$ and set $L_+ \coloneqq  L \cup \{l\}, \, L'_- \coloneqq L' \setminus \{l\}, \, L''_+ \coloneqq L'' \cup \{l\}$. We shall prove
\begin{align} \label{numer}
\| {\mathcal F}_{L, L',L''} \|_{\ell^1(\Z^d)} \leq
C_r \big( \| {\mathcal F}_{L_+, L'_-,L''} \|_{\ell^1(\Z^d)} + \| {\mathcal F}_{L, L'_-,L''_+}\|_{\ell^1(\Z^d)} \big)
\end{align}
with $C_r = {\bf C}_{\rm diag}(r)(1+{\bf C}_{\rm UP}(r))$ using Lemma~\ref{L3} and Proposition~\ref{UP}. This gives
\begin{align}
\label{IWeq:9}
\| {\mathcal F}_{\emptyset, [K], \emptyset} \|_{\ell^1(\Z^d)} 
\leq C_r^K
\sum_{L \subseteq [K]}
\| {\mathcal F}_{L,\emptyset,[K] \setminus L} \|_{\ell^1(\Z^d)}
\end{align}
after $K$ iterations of \eqref{numer}. The implied constant in \eqref{IWeq:9} is acceptable because $K \leq K_N$.

Regarding \eqref{numer}, by Lemma~\ref{lem:fact} we can expand $\| {\mathcal F}_{L, L',L''} \|_{\ell^1(\Z^d)}$ as follows
\[
\sum_{x \in \Z^d} \sum_{q \in {\mathcal S}_{L}} 
\Big( 
\sum_{q^* \in S_l}
\sum_{q' \in {\mathcal S}_{L'_-}}  
\sum_{q'' \in {\mathcal S}_{L''}} 
\sum_{u'' \in\mathcal I(q'')} \big \|
f^\circ_{\mathcal I(q) \oplus \mathcal I(q^*) \oplus \mathcal I(q') +u''}(x) \big\|_{H}^{2} \Big)^r.
\]
Applying Lemma~\ref{L3} to the sum over $S_l$, we see that $\| {\mathcal F}_{L, L',L''} \|_{\ell^1(\Z^d)}$ is controlled by the diagonal sum, which is equal to $\| {\mathcal F}_{L_+, L'_-,L''} \|_{\ell^1(\Z^d)}$, and the off-diagonal sum corresponding to distinct $q^*_1,\ldots, q^*_r \in S_l$. To control the latter by $\| {\mathcal F}_{L, L'_-,L''_+}\|_{\ell^1(\Z^d)}$, we expand the $r$-fold product and, for fixed $q \in {\mathcal S}_L, \, q'_i \in {\mathcal S}_{L'_-}, \, q''_i \in {\mathcal S}_{L''}, \, u''_i \in \mathcal I(q_i'')$ and distinct $q^*_i \in S_l$, we set $F_{u^*_i} \coloneqq f^\circ_{\mathcal I(q) \oplus \mathcal I(q_i') + u^*_i +u_i''}$, and we claim that Proposition~\ref{UP} implies
\begin{align} \label{numer2}
\sum_{x \in \Z^d}
\prod_{i \in [r]} 
\Big \|
\sum_{u^*_i \in\mathcal I(q^*_i)}
F_{u^*_i}(x) \Big\|_{H}^{2}
\leq {\bf C}_{\rm UP}(r)
\sum_{x \in \Z^d}
\prod_{i \in [r]}
\sum_{u^*_i \in\mathcal I(q^*_i)} 
\|F_{u^*_i}(x)\|_{H}^{2}.
\end{align}
Having established \eqref{numer2}, we can complete the proof of \eqref{numer} by replacing the off-diagonal sum with the full sum and then combining the appropriate sums into $r$-th powers. 

Turning to the proof of \eqref{numer2}, we claim that, for every $2r$-tuple $(u^*_{1,1}, u^*_{1,2}, \dots, u^*_{r,1}, u^*_{r,2}) \in \mathcal I(q^*_{1})^2 \times \cdots \times \mathcal I(q^*_{r})^2$ with the uniqueness property, the Fourier transform of the function
\[
\prod_{i \in [r]} \langle F_{u^*_{i,1}}, F_{u^*_{i,2}} \rangle_{H}
\]
vanishes in a neighborhood of the origin so that \eqref{IWeq:6} holds. Indeed, since $H$ is separable, the Fourier transform of this function is a finite sum of functions, each supported on a set
\[
2r \varepsilon {\bf Q}^d +
(u^*_{1,1} - u^*_{1,2} + v_{1,1} - v_{1,2})
+ \dots +
(u^*_{r,1} - u^*_{r,2} + v_{r,1} - v_{r,2}),
\]
where $u^*_{i,1} + v_{i,1}, \, u^*_{i,2} + v_{i,2}$ are fractions that, in reduced form, have denominators at most $N$. We rewrite the above set as $2r \varepsilon {\bf Q}^d + v$ and observe that $v \notin \Z^d$ because, due to the uniqueness property, property $\mathcal O_N$, and $\mathbb P$-irreducibility, for at least one fraction $u^*_{i,1} - u^*_{i,2} + v_{i,1} - v_{i,2}$ the corresponding factor $q^*_i \in S_l$ cannot be completely eliminated from the denominator of $u^*_{i,1} - u^*_{i,2} + v_{i,1} - v_{i,2}$, hence the same is true for $v$ by property $\mathcal O_N$ because $q^*_i$ are distinct. Thus, for some $i \in [d]$, the distance between $\Z$ and the $i$-th coordinate of $v$ is at least $N^{-2r}$. By $\varepsilon \leq (4r N^{2r})^{-1}$, this implies $\Z^d \cap (2r \varepsilon {\bf Q}^d +v) = \emptyset$. Since condition \eqref{IWeq:6} is satisfied, an application of Proposition~\ref{UP} completes the proof of \eqref{numer2}.

\subsection{\textbf{Square function estimates}}
Gathering \eqref{denom-0}, \eqref{denom-II}, \eqref{denom-III}, and \eqref{IWeq:9} together, we see that the verification of \eqref{IW UPPER} has been reduced to proving
\begin{align*}
\sum_{L \subseteq [K]}
\| {\mathcal F}_{L,\emptyset,[K] \setminus L} \|_{\ell^1(\Z^d)} \lesssim_{r,d}  N^{2rC / \log \log N}
\mathbf{C}_{G}(\mathcal R^d_{\le N})^{2r} \|f\|_{\ell^{2r}(\Z^{d};H)}^{2r}.
\end{align*}
Since there are $2^K$ subsets $L \subseteq [K]$, we can deal with each term in the sum above separately at the expense of introducing a factor $C_r 2^{C_r \log N / \log \log N}$ which is acceptable.

We fix $L \subseteq [K]$ and set $\mathcal I' \coloneqq \mathcal I(\mathcal S_{[K] \setminus L}) \cap \mathcal R^d_{\le N}$. Since $\ind{N,K,j}(u+u') = 0$ for $u \in \mathcal I(q)$ and $q \in S_L$ when $u' \in \mathcal I(\mathcal S_{[K] \setminus L}) \setminus \mathcal R^d_{\le N}$, we are reduced to showing the inequality
\begin{align}
\label{IWeq:10}
\sum_{x \in \Z^d} 
\sum_{q \in \mathcal S_L}
\Big( 
\sum_{u' \in \mathcal I'}
\big \| 
f^\circ_{\mathcal I(q) +u'}(x) \big\|_{H}^{2} \Big)^r
\lesssim_{r,d} N^{2r C /\log \log N} \mathbf{C}_{G}(\mathcal R^d_{\le N})^{2r}
\|f\|_{\ell^{2r}(\Z^{d};H)}^{2r},
\end{align}
which we view as a linear operator norm bound from $\ell^{2r}(\Z^{d};H)$ to $\ell^{2r}(\Z^{d} \times \mathcal S_{L};\ell^{2}( \mathcal I' ;H))$. The reduction to \eqref{IWeq:10} uses the fact that $\mathcal I(q) \cap \mathcal I(q') = \emptyset$ holds for distinct $q, q' \in \mathcal S_{[K]\setminus L}$.

The following lemma is the key improvement, resulting in a factor of $1 / \log \log N$ in the exponent rather than the $\log \log \log N / \log \log N$ factor used in \cite[Theorem~3.3]{KMPWW}.

\begin{lemma} \label{L4}
	There exists $\mathbf C_{\rm div} \in \RR_+$ such that, for each $q \in \Z_+$ and $N \in \NN_{\geq N_0}$, if $q$ has $K \in [K_N]$ prime divisors, then $q$ has at most $M_N \coloneqq \mathbf C_{\rm div} N^{\mathbf C_{\rm div} / \log \log N}$ divisors $b \in [N]$. 
\end{lemma}

\begin{proof}
	Let $\{p_1, \dots, p_K\}$ be the set of prime divisors of $q$. We shall estimate the number of tuples $(\alpha_1,\ldots, \alpha_K) \in \NN^K$ such that $p_1^{\alpha_1}\cdots p_K^{\alpha_K}\le N$. We may take $K = K_N = \lfloor \frac{5 \log N}{\log \log N} \rfloor$ and the smallest $K$ prime numbers $p_1 < \dots < p_K$. Note that $p_1^{\alpha_1}\cdots p_K^{\alpha_K}\le N$ implies
    \[
    \alpha_1 \log p_1 + \dots + \alpha_K \log p_K = \log(p_1^{\alpha_1}\cdots p_K^{\alpha_K}) \leq \log N
    \]
    and so $\alpha_k \leq \log N / \log p_k$ for each $k \in [K]$. We partition $\NN_{\le \log N / \log p_k}$ into the intervals
	\[
	I_{k,i} \coloneqq \NN_{\le \log N / \log p_k} \cap \big[(i-1) \log \log N / \log p_k,\, i\log \log N / \log p_k\big).
	\]  
    If $N$ is large, then the prime number theorem implies $p_k \leq 2 K \log K \leq 20 \log N$ so that $|I_{k,i}| \leq 1 + \log \log N / \log p_k \leq 3 \log \log N / \log p_k$ for all $k \in [K]$ and $i \in \Z_+$. Consider
	\[
	\NN_{\le \log N / \log p_1} \times \cdots \times 
	\NN_{\le \log N / \log p_K}
	\ni (\alpha_1,\ldots, \alpha_K)
	\mapsto (i_1, \dots, i_K) \in \Z_+^K,
	\]
	where $(i_1, \dots, i_K)$ is the unique tuple such that
	$(\alpha_1,\ldots, \alpha_K)\in I_{1, i_1}\times\cdots\times I_{K, i_K}$. Then
	\[
	(i_1-1) \log \log N + \dots + 
	(i_K-1) \log \log N 
	\leq \alpha_1 \log p_1 + \dots + \alpha_K \log p_K \leq \log N 
	\]
	so that $i_1+ \cdots +i_K \leq K + \log N / \log \log N \leq 6 \log N / \log \log N$. Let $\mathbb I$ denote the set of all tuples $(i_1, \dots, i_K) \in \Z_+^K$ that satisfy this restriction. Then
	\begin{align} \label{i-tuples}
	|\mathbb I| \leq \binom{\lfloor 6 \log N / \log \log N \rfloor + K}{K}
	\leq 2^{\lfloor 6 \log N / \log \log N \rfloor + K}
	\le CN^{C/ \log \log N}. 
	\end{align}
	Next, for each $(i_1, \dots, i_K) \in \mathbb I$, note that $|I_{1, i_1}\times\cdots\times I_{K, i_K}|$ does not exceed
	\[
	(3 \log \log N / \log p_1) \cdots 
	(3 \log \log N / \log p_K)
	= (3 \log \log N)^K / (\log p_1 \cdots \log p_K). 
	\]
	To estimate the above, we use the prime-counting function $\pi(x) \coloneqq |\mathbb P \cap \RR_{\leq x}|$. We have
	\[
	\log (\log p_1 \cdots \log p_K)
	= \lim_{\delta \to 0^+} \int_{2-\delta}^{p_K + \delta} \log \log x \, {\rm d} \pi(x) = K \log \log p_K
	- \int_2^{p_K} \frac{\pi(x)}{x \log x} \, {\rm d}x
	\]
	with the last integral bounded by $C p_K / \log^2 p_K$. By $K \leq p_K \leq 2 K \log K$, this implies
	\[
	\log (\log p_1 \cdots \log p_K) \geq
	K \log \log K - 2C K / \log K 
    \geq K \log \log K - K,
	\]
    provided that $N$, and hence $K$, is large. Thus, for each $(i_1, \dots, i_K) \in \mathbb I$, we have
	\begin{align} \label{alpha-tuples}
    |I_{1, i_1}\times\cdots\times I_{K, i_K}| \leq (3 e \log \log N / \log K)^K \leq C^K \leq C N^{C / \log \log N}.
	\end{align}
	Combining \eqref{i-tuples} and \eqref{alpha-tuples} completes the proof.
\end{proof}

\begin{definition}[Admissible divisors]
We say that $b \in \Z_+$ is an \emph{admissible divisor} of $q \in \mathcal S_L$ if $b \in [N]$ and $Q_b = q$. Let $\mathcal B_q$ denote the set of all admissible divisors of $q \in \mathcal S_L$.
\end{definition}

Using Lemma~\ref{L4}, we partition $\mathcal I(q) \cap \mathcal R^d_{\le N}$ into disjoint sets $\mathcal R(b_1), \dots,\mathcal R(b_{M})$ for some $M \in [M_N]$ and $b_1, \dots, b_{M} \in \mathcal B_q$. Note that the set $\mathcal I(q)$ appearing on the left-hand side of \eqref{IWeq:10} can be replaced with $\mathcal I(q) \cap \mathcal R^d_{\le N}$. Hence, denoting $F_{\mathcal I}(x)\coloneqq \big( \sum_{u' \in \mathcal I'} \| f^\circ_{\mathcal I+u'}(x) \|_{H}^{2} \big)^r$ for $\mathcal I \subseteq \mathcal I(q)$, we can bound the left-hand side of \eqref{IWeq:10} as follows
\begin{align*}
\sum_{x \in \Z^d} 
\sum_{q \in \mathcal S_L}
F_{\mathcal I(q)}(x)\le M_N^{2r-1}
\sum_{q \in \mathcal S_L}\sum_{b \in \mathcal B_q}
\|F_{\mathcal R(b)}\|_{\ell^1(\Z^d)}
\le M_N^{2r} 
\sum_{q \in \mathcal S_L}\sup_{b \in \mathcal B_q}
\|F_{\mathcal R(b)}\|_{\ell^1(\Z^d)}
\end{align*}
using the triangle inequality and H{\"o}lder's inequality. Let $f_{\mathcal I} \coloneqq T_{\Z^d}\big[\llbracket G \mid \mathfrak m\rrbracket_{\mathcal I}\big] f$ be a variant of $f^\circ_{\mathcal I}$ without the factor $\ind{N,K,j}$. By linearizing the supremum, we are reduced to proving
\begin{align*}
\sum_{x \in \Z^d} 
\sum_{q \in {\mathcal S}_L}
\Big( 
\sum_{u' \in \mathcal I'}
\big \|  
 f_{\mathcal R(b_q)+u'} (x) \big\|_{H}^{2} \Big)^r
\lesssim_{r,d} N^{2r C  / \log \log N} \mathbf{C}_{G}(\mathcal R^d_{\leq N})^{2r}
\|f\|_{\ell^{2r}(\Z^{d};H)}^{2r}
\end{align*}
uniformly in all possible choices of admissible divisors $b_q \in \mathcal B_q$. Indeed, this follows because $\ind{N,K,j}$ is constant on each set $\mathcal R(b_q) + u'$ and vanishes on $\mathcal I(q) \setminus \mathcal R^d_{\le N}$.

Now, for each fixed choice of admissible divisors $b_q$ of $q\in {\mathcal S}_L$, we shall obtain
\begin{align}\label{IWeq:again}
\sum_{x \in \Z^d} 
\sum_{q \in {\mathcal S}_L}
\Big( 
\sum_{u' \in \mathcal I'}
\big \|  
 f_{\mathcal R(b_q)+u'} (x) \big\|_{H}^{2} \Big)^r
\lesssim_{r,d} 
\sum_{x \in \Z^d} 
\sum_{q \in {\mathcal S}_L}
\Big( 
\sum_{u' \in \mathcal I'}
\big \|  
 f^{\zeta_\varepsilon}_{\mathcal R(b_q)+u'} (x) \big\|_{H}^{2} \Big)^r,
\end{align}
where $f_{\mathcal I}^{\zeta_\varepsilon} \coloneqq T_{\Z^d} \big[\llbracket G \mid {\zeta_\varepsilon}\rrbracket_{\mathcal I}\big] f$ is a variant of $f_{\mathcal I}$ with $\mathfrak m$ replaced by a bump multiplier ${\zeta_\varepsilon}$.  We shall use the vector-valued Marcinkiewicz--Zygmund theorem. Indeed, once $q \in {\mathcal S}_{L}$ is fixed, both sides of \eqref{IWeq:again} are $2r$-th powers of norms of some $\ell^{2}(\mathcal I';H)$-valued functions.
\begin{theorem}
	\label{MZ}
	Fix $r \in \Z_+$. There exists ${\bf C}_{\rm MZ}(r) \in \RR_+$ such that, for each $\sigma$-finite space $(X,\mathcal B(X),\mu)$, if $T$ is a linear operator satisfying $\| T \|_{L^{2r}(X;H) \to L^{2r}(X;H)} \leq 1$, then
	\begin{align*}
	\Big \| \Big( \sum_{n\in \Z_+}  \| TF_{n}\|_{H}^{2} \Big)^{1/2} \Big\|_{L^{2r}(X,\mu)}
	\leq {\bf C}_{\rm MZ}(r)
	\Big\| \Big( \sum_{n\in \Z_+} \| F_{n} \|_{H}^{2} \Big)^{1/2} \Big\|_{L^{2r}(X,\mu)}.
	\end{align*}
\end{theorem}
\begin{proof}
A proof can be found in \cite[Theorem~2.48]{MSZ3}. 
\end{proof}
Recall that $\mathfrak m$ is supported on $\varepsilon \mathbf{Q}^d$ and $ \varepsilon \leq (4r N^{2r})^{-1}$ holds by \eqref{eq:support}. Let \( \phi \colon \RR^d \to [0,1] \) be a~fixed smooth function supported on $\frac{5}{3}\mathbf{Q}^d$ and identically equal to $1$ on $\frac{4}{3}\mathbf{Q}^d$. Let $\psi$ be a~fixed nonnegative smooth function supported on $\frac{1}{3} \mathbf{Q}^d$, normalized by $\int_{\RR^d} \psi(x)\, \mathrm{d}x = 1$.
	
	We define $\zeta  \coloneqq  \phi * \psi$. Note that $\zeta \colon \RR^d \to [0,1]$ is a smooth function supported on $2 \mathbf{Q}^d$ and identically equal to $1$ on $\mathbf{Q}^d$. We also define $\Phi_\varepsilon(x)  \coloneqq  \Phi(x/\varepsilon)$ for $\Phi \in \{\phi, \psi, \zeta\}$.
    
    Fix $q \in {\mathcal S}_{L}$ and $u' \in \mathcal I'$. We use the function $\zeta_\varepsilon$ to obtain the following factorization
\[
\llbracket G \mid \mathfrak m\rrbracket_{\mathcal R(b_q) + u'}
=  \llbracket 1 \mid \mathfrak m\rrbracket_{\mathcal Q(b_q) + u'}\cdot
\llbracket G \mid \zeta_{\varepsilon}\rrbracket_{\mathcal R(b_q) + u'}.
\]
The multiplier $\llbracket 1 \mid \mathfrak m\rrbracket_{\mathcal Q(b_q) + u'}$ corresponds, up to modulation determined by $u'$, to the operator $T_{\Z^d}\big[\llbracket 1 \mid \mathfrak m\rrbracket_{\mathcal Q(b_q)}\big]$ which is bounded on $\ell^{2r}(\Z^{d};H)$. In fact, by $b_q \leq N$ and $\varepsilon \leq (4r N^{2r})^{-1}$, we can use  Theorem~\ref{prop:msw} to bound its norm by ${\bf C}_{\rm MSW}(d)$. By modulation symmetry of $\mathcal F_{\Z^d}$ and Theorem~\ref{MZ} applied with $f^{\zeta_\varepsilon}_{\mathcal R(b_q)+u'}$ in place of $F_n$, we obtain \eqref{IWeq:again}.

\subsection{Square function estimates for the nice bump multiplier}
Finally, for a fixed choice of admissible divisors $b_q \in \mathcal B_q$ and for each $s\in[1, \infty]$, we prove the following inequality
\begin{align}
\label{IWeq:13}
\bigg [ \sum_{x \in \Z^d} 
\sum_{q \in \mathcal S_L}
\Big( 
\sum_{u' \in \mathcal I' }
\big \| 
f^{\zeta_\varepsilon}_{\mathcal R(b_q)+u'}(x) \big\|_{H}^{2} \Big)^s \bigg ]^{1/2s}
\lesssim_d N^{C / \log \log N} \mathbf{C}_{G}(\mathcal R^d_{\leq N})
\|f\|_{\ell^{2s}(\Z^{d};H)},
\end{align}
which we view as a linear operator norm bound from $\ell^{2s}(\Z^{d};H)$ to $\ell^{2s}(\Z^{d} \times \mathcal S_{L};\ell^{2}( \mathcal I' ;H))$. By interpolation, it suffices to prove \eqref{IWeq:13} for $s \in \{1, \infty\}$. 

\smallskip \paragraph{\bf Case $s=1$} Each $\llbracket G \mid \zeta_{\varepsilon}\rrbracket_{u + u'}$ is supported on $u + u' + 2\varepsilon {\bf Q}^d$. We claim that these sets are pairwise disjoint. Indeed, for distinct $q, q'\in {\mathcal S}_L$, the corresponding admissible divisors $b_{q}, b_{q'} \in [N]$ are also distinct, since $q=Q_{b_q} \neq Q_{b_{q'}}=q'$. Hence, the rational numbers $u + u'$ appearing in \eqref{IWeq:13} do not repeat. By \eqref{eq:support} and $r \geq 2$, we obtain $\varepsilon \leq (4r N^{2r})^{-1} \le (8 N^{4})^{-1}$ and the claim follows, since the denominators of $u + u'$ do not exceed $N^2$. By Plancherel's theorem \eqref{IWeq:13} holds with $\mathbf{C}_{G}(\mathcal R^d_{\leq N})$ in place of $N^{C / \log \log N} \mathbf{C}_{G}(\mathcal R^d_{\leq N})$.

\smallskip \paragraph{\bf Case $s=\infty$} 
For every $b_q \in \mathcal B_q$, we have the following inclusion--exclusion formula
\[
\sum_{u\in \mathcal R(b_q)} F(u)
= \sum_{b \in \Z_+ : \, b|b_q} \mu(b_q / b) \sum_{u \in \mathcal Q(b)} F(u),
\] 
where $\mu \colon \Z_+ \to \{-1,0,1\}$ is the M\"obius function. Now \eqref{IWeq:13} for $s=\infty$ follows from this decomposition with $F(u)=f^{\zeta_\varepsilon}_{u+u'}$ for $u\in \mathcal R(b_q)$ and Lemma~\ref{L4}, if only we can show
\begin{align*}
\Big( \sum_{u' \in \mathcal I' } 
\big\| f^{\zeta_\varepsilon}_{\mathcal Q(b)+u'}(x)
\big\|_{H}^2 \Big)^{1/2}
\lesssim_d \mathbf{C}_{G}(\mathcal R^d_{\leq N})
\|f\|_{\ell^{\infty}(\Z^{d};H)}
\end{align*}
with the implicit constant controlled uniformly in $x \in \Z^d$ and $b \in \bigcup_{q \in \mathcal S_L} \mathcal B_q$. By translation invariance, we can consider $x=0$. By duality, it suffices to prove
\begin{align} \label{5-squarefunction:3}
\Big\| 
\sum_{u' \in \mathcal I' } h_{u'} \,
\mathcal F_{\Z^d}^{-1} \big(
\llbracket G \mid \zeta_{\varepsilon}\rrbracket_{\mathcal Q(b) + u'}
\big) \Big\|_{\ell^1(\Z^{d};H)}
\lesssim_d \mathbf{C}_{G}(\mathcal R^d_{\leq N}) \Big( \sum_{u' \in \mathcal I'} \|h_{u'}\|_H^2 \Big)^{1/2}
\end{align}
for any sequence $(h_{u'} : u'\in \mathcal I')\subseteq H$. 

\begin{proof}[Proof of inequality \eqref{5-squarefunction:3}]
For ease of exposition, we will proceed in a few steps.

\smallskip \paragraph{\bf Step~1} We split the left-hand side of \eqref{5-squarefunction:3} as follows
\begin{equation} \label{N-squarefunction:4} 
		\sum_{a \in [b]^d} 
		\Big\| \sum_{u' \in \mathcal I' } h_{u'} \,
        \mathcal F_{\Z^d}^{-1}
        \big( \llbracket G \mid \zeta_{\varepsilon}\rrbracket_{\mathcal Q(b) + u'}
        \big) 
		\Big\|_{\ell^1(b\mathbb{Z}^d+a; H)}
\end{equation}
and notice that, for each $u\in \mathcal Q(b)$ and $u' \in \mathcal I'$, the denominators of $u$ and $u'$ are coprime. Hence, by condition~\ref{cond:A} and the fact that translation in the Fourier variable corresponds to modulation in the
spatial variable, the integrand in
\eqref{N-squarefunction:4} evaluated at $bx+a$ is equal to
\begin{equation*}
    \sum_{u' \in \mathcal I'} h_{u'}
    \sum_{u\in \mathcal Q(b)}
    G(u)\,e\big((bx+a) \cdot u \big) \,
    \mathcal{F}_{\mathbb{Z}^d}^{-1} \big(\llbracket G \mid \zeta_{\varepsilon}\rrbracket_{u'} \big)(bx+a).
\end{equation*}
Since the denominator of each $u \in\mathcal{Q}(b)$ divides $b$, we can rewrite this as
\begin{equation*}
    \sum_{u \in\mathcal{Q}(b)}G(u)\,e(a \cdot u)\,
    \sum_{u' \in \mathcal I'}
    h_{u'}
    \, 
    \mathcal{F}_{\mathbb{Z}^d}^{-1}\big(
\llbracket G \mid \zeta_{\varepsilon}\rrbracket_{u'}
    \big)(bx+a).
\end{equation*}
Consequently, for each fixed $a\in[b]^d$, the corresponding $\ell^1$ term in \eqref{N-squarefunction:4} is equal to
\begin{equation}\label{N-squarefunction:5}
    \Big|\sum_{u \in\mathcal{Q}(b)}G(u)\,e(a \cdot u) \Big|\,
    \Big\| \sum_{u' \in \mathcal I'}
    h_{u'} \, \mathcal{F}_{\mathbb{Z}^d}^{-1}
    \big(\llbracket G \mid \zeta_{\varepsilon}\rrbracket_{ u'}
    \big)
    \Big\|_{\ell^1(b\mathbb{Z}^d+a;H)}.
\end{equation}
		
\smallskip \paragraph{\bf Step~2} 
Noting that $\zeta_\varepsilon = \varepsilon^{-d} \psi_\varepsilon * \phi_\varepsilon$, we can write $\llbracket G \mid \zeta_{\varepsilon}\rrbracket_{u'} = \varepsilon^{-d} \, \psi_\varepsilon * \llbracket G \mid \phi_{\varepsilon}\rrbracket_{u'}$ for each $u' \in \mathcal I'$. Since the Fourier transform intertwines convolution with pointwise multiplication, we can apply the Cauchy--Schwarz inequality to the $\ell^1$ norm in \eqref{N-squarefunction:5} to obtain the estimate
		\begin{equation} \label{N-squarefunction:5.0} \varepsilon^{-d} \| \mathcal{F}_{\mathbb{Z}^d}^{-1} \psi_\varepsilon \|_{\ell^2(b\mathbb{Z}^d+a)} \cdot \Big \| \sum_{u' \in \mathcal I'} h_{u'} \, \mathcal{F}_{\mathbb{Z}^d}^{-1} \big(\llbracket G \mid \phi_{\varepsilon}\rrbracket_{u'} \big) \Big \|_{\ell^2(b\mathbb{Z}^d+a; H)}.
		\end{equation}
In the next two steps we  show that both integrands above are, in the sense of their $\ell^2$ norms, equally distributed among the residue classes, which will be essential to establish \eqref{5-squarefunction:3}.

\smallskip \paragraph{\bf Step~3} 
Define $\psi_{\varepsilon, b, a} \colon \TT^d \to \C$ by
        \[
        \psi_{\varepsilon, b, a}(\xi) \coloneqq
        b^{-d} \sum_{\beta \in \mathcal Q(b)} \psi_{\varepsilon}(\xi+\beta) \, e \big(a \cdot (\xi+\beta) \big).
        \]
Substituting $\zeta = \xi + \beta$, we observe that 
        \[
        \mathcal{F}_{\mathbb{Z}^d}^{-1} \psi_{\varepsilon, b, a}(n)
        = b^{-d} \sum_{\beta \in \mathcal Q(b)} \int_{\TT^d} \psi_{\varepsilon}(\zeta) \, e(a \cdot \zeta) \, e \big(n \cdot (\zeta-\beta) \big) \, {\rm d}\zeta = \ind{b\Z^d}(n) \cdot \mathcal{F}_{\mathbb{Z}^d}^{-1} \psi_\varepsilon (n+a). 
        \]
Thus, by Plancherel's theorem and \eqref{eq:support}, we obtain
        \begin{align}
        \label{eq:2}
        \| \mathcal{F}_{\mathbb{Z}^d}^{-1} \psi_\varepsilon \|_{\ell^2(b\mathbb{Z}^d+a)} =
        \| \mathcal{F}_{\mathbb{Z}^d}^{-1} \psi_{\varepsilon, b, a} \|_{\ell^2(\mathbb{Z}^d)}
        = b^{-d/2} \| \psi_{\varepsilon} \|_{\ell^2(\TT^d)}.
        \end{align}

\smallskip \paragraph{\bf Step~4}
Analogously, define $\phi_{\varepsilon, b, a} \colon \TT^d \to \C$ by
        \[
        \phi_{\varepsilon, b, a}(\xi) \coloneqq
        b^{-d} \sum_{u' \in \mathcal I'} \sum_{\beta \in \mathcal Q(b)} 
        h_{u'} \, G(u') \, \phi_{\varepsilon}(\xi+\beta-u') \, e \big(a \cdot (\xi+\beta) \big).
        \]
Then, as in the previous step, we obtain
        \begin{align*}
        \mathcal{F}_{\mathbb{Z}^d}^{-1} \phi_{\varepsilon, b, a}(n)=        \ind{b\Z^d}(n) \cdot\sum_{u' \in \mathcal I'}
        h_{u'} \, \mathcal{F}_{\mathbb{Z}^d}^{-1} 
		\big( \llbracket G \mid \phi_{\varepsilon}\rrbracket_{u'}\big)(n+a),
        \end{align*}
and consequently, by Plancharel's theorem, we may further write
        \begin{align}
        \label{eq:3}
        \Big \| \sum_{u' \in \mathcal I'}
        h_{u'} \, \mathcal{F}_{\mathbb{Z}^d}^{-1} 
		\big( \llbracket G \mid \phi_{\varepsilon}\rrbracket_{u'}\big)
		\Big \|_{\ell^2(b\mathbb{Z}^d+a; H)}
        = b^{-d/2} \Big \| \sum_{u' \in \mathcal I'}
        h_{u'} \, \llbracket G \mid \phi_{\varepsilon}\rrbracket_{u'} \Big \|_{\ell^2(\TT^d; H)}
        \end{align}
after observing that the pieces of $\phi_{\varepsilon, b, a}$ do not overlap. Indeed, any two distinct elements $u' + \beta$ with $u' \in \mathcal I'$ and $\beta \in \mathcal{Q}(b)$ are separated in the $\ell^\infty$ metric by at least $N^{-4}$, while \eqref{eq:support} and $r \geq 2$ imply $\varepsilon \leq (4r N^{2r})^{-1} \leq (8N^4)^{-1}$, which guarantees \eqref{eq:3}.

\smallskip \paragraph{\bf Step~5}
Using \eqref{eq:2} and \eqref{eq:3}, we conclude that \eqref{N-squarefunction:5.0} is equal to
        \[
        \varepsilon^{-d} b^{-d/2} \| \mathcal{F}_{\mathbb{Z}^d}^{-1} \psi_\varepsilon \|_{\ell^2(\mathbb{Z}^d)} \cdot b^{-d/2} \Big \| \sum_{u' \in \mathcal I'}
        h_{u'} \, \mathcal{F}_{\mathbb{Z}^d}^{-1}
		\big(\llbracket G \mid \phi_{\varepsilon}\rrbracket_{u'}
		\big) 
		\Big \|_{\ell^2(\mathbb{Z}^d; H)}.
        \]
Thus, by Plancherel's theorem and condition~\ref{cond:B}, we can control \eqref{N-squarefunction:5} by
        \[
        C(d)
        \Big|\sum_{u \in\mathcal{Q}(b)}G(u)\,e(a \cdot u) \Big| \cdot b^{-d} \|G\|_{\ell^\infty(\mathcal R^d_{\leq N})} 
		\Big( \sum_{u' \in \mathcal I'} \|h_{u'}\|_H^2 \Big)^{1/2},
        \]
since, as in the previous step, \eqref{eq:support} implies that the functions $\llbracket G \mid \phi_{\varepsilon}\rrbracket_{u'}$ do not overlap when $u'$ runs over the set $\mathcal I'$. Summing this bound over the residue classes $a\in[b]^d$ yields the desired inequality \eqref{5-squarefunction:3}. This also completes the proof of Theorem~\ref{thm:IW:upper}.	
\end{proof}
   
\section{Proof of Theorem~\ref{thm:IW:lower}: the lower bound} \label{sec:lower}

In this section, we present a detailed proof of Theorem~\ref{thm:IW:lower}, illustrating that the dependence on $N$ in Theorem~\ref{thm:IW:upper} cannot be eliminated. This contrasts with Tao's result \cite{TaoIW} that establishes the Ionescu--Wainger multiplier theorem for the set of Ionescu--Wainger fractions with uniform bounds with respect to the size of the underlying family of fractions.

Our goal is to prove that, for $p \in (1,\infty) \setminus \{2\}$ and infinitely many $N \in \Z_+$, we have
\begin{equation} \label{eq:N-c_p}
     \big \| T_{\Z} \big[\llbracket 1 \mid {\mathfrak m}_{\varepsilon_N}\rrbracket_{\mathcal R_{\leq N}}\big] \big \|_{\ell^{p}(\Z) \to \ell^{p}(\Z)}
     \gtrsim_{p}
     N^{(\log N)^{-c_p}}
\end{equation}
with some $0 < \varepsilon_N \leq \vartheta_N$.
\begin{proof}[Proof of inequality \eqref{eq:N-c_p}]
Without loss of generality, we can assume that $(\vartheta_N)_{N\in \mathbb{Z}_+}$ is nonincreasing. We can work with  $\mathcal R(N)=\mathcal R_{\leq N}\setminus \mathcal R_{\leq N-1}$ in place of $\mathcal R_{\leq N}$ in \eqref{eq:N-c_p}, since
\[
\bigl\| T_{\Z} \big[\llbracket 1 \mid {\mathfrak m}_{\varepsilon_N}\rrbracket_{\mathcal R(N)}\big] \bigr\|_{\ell^{p}(\Z) \to \ell^{p}(\Z)}
\le\sum_{J\in\{N-1, N\}}
\bigl\| T_{\Z}\big[\llbracket 1 \mid {\mathfrak m}_{\varepsilon_N}\rrbracket_{\mathcal R_{\leq J}}\big] \bigr\|_{\ell^{p}(\Z) \to \ell^{p}(\Z)}.
\]
By duality, we can assume that $p\in(1,2)$ in \eqref{eq:N-c_p}. We will proceed in a few steps.

\smallskip \paragraph{\bf Step~1}
Let $\varphi \colon \R \to \R$ be a smooth function such that $\ind{{\bf Q}} \le \varphi \le \ind{2{\bf Q}}$. Then $\varphi \cdot \mathfrak m = \mathfrak m$. Setting  $\varphi_{\varepsilon_N}(x) \coloneqq \varphi(x/{\varepsilon_N})$, and taking $\varepsilon_N \in (0,1/4N)$, we have 
\begin{equation*}
\llbracket 1 \mid {\mathfrak m}_{\varepsilon_N}\rrbracket_{\mathcal R(N)}
    =
\llbracket 1 \mid {\mathfrak m}_{\varepsilon_N}\rrbracket_{\mathcal R(N)} \cdot
\llbracket 1 \mid \varphi_{\varepsilon_N}\rrbracket_{\mathcal Q(N)}.
\end{equation*}
Thus, for every $f\in \ell^p(\Z)$, we have
\begin{equation*}
    \big \| T_{\Z}\big[\llbracket 1 \mid {\mathfrak m}_{\varepsilon_N}\rrbracket_{\mathcal R(N)}\big] f \big \|_{\ell^p(\Z)}
    \le
    \big \|  T_{\Z}\big[\llbracket 1 \mid {\mathfrak m}_{\varepsilon_N}\rrbracket_{\mathcal R(N)}\big]\big\|_{\ell^p(\Z)\to\ell^p(\Z)}
    \big \| T_{\Z}\big[\llbracket 1 \mid \varphi_{\varepsilon_N}\rrbracket_{\mathcal Q(N)}\big] f \big\|_{\ell^p(\Z)}.
\end{equation*}
In particular, taking $f$ to be $\delta_0 \coloneqq \ind{\{0\}}$, we obtain the lower bound
\begin{equation}\label{eq:lower bound 1}
    \big \|  T_{\Z}\big[\llbracket 1 \mid {\mathfrak m}_{\varepsilon_N}\rrbracket_{\mathcal R(N)}\big]\big\|_{\ell^p(\Z)\to\ell^p(\Z)}
    \ge
    \frac{
           \big \| T_{\Z}\big[\llbracket 1 \mid {\mathfrak m}_{\varepsilon_N}\rrbracket_{\mathcal R(N)}\big] \delta_0 \big \|_{\ell^p(\Z)}
    }{
            \big \| T_{\Z}\big[\llbracket 1 \mid \varphi_{\varepsilon_N}\rrbracket_{\mathcal Q(N)}\big] \delta_0 \big\|_{\ell^p(\Z)}
    }.
\end{equation}

\smallskip \paragraph{\bf Step~2} We now carry out some calculations that will be used later. Fix a $1$-periodic weight $G\colon \QQ \to \C$ and a smooth function $\eta \colon \R \to \R$ supported on ${\bf Q}$. For every $N \in \NN$, we have
\begin{equation}\label{ce:eq1}
    \mathcal{F}_{\Z}^{-1} \big( \llbracket G \mid \eta\rrbracket_{\mathcal R(N)} \big)(n)
    = \mathcal{F}_{\Z}^{-1}(\eta)(n)
    \sum_{u \in \mathcal R(N)}
    G(u)\,e(n \cdot u).
\end{equation}
Analogously, for $G \equiv 1$ and $\mathcal R(N)$ replaced by $\mathcal Q(N)$, we have 
\begin{equation}\label{ce:eq2}
    \mathcal{F}_{\Z}^{-1} \big( \llbracket 1 \mid \eta\rrbracket_{\mathcal Q(N)} \big)(n)
    = N\,\mathds{1}_{N\mid n}\,\mathcal{F}_{\Z}^{-1}(\eta)(n).
\end{equation}

\smallskip \paragraph{\bf Step~3} We apply \eqref{ce:eq2} 
and, for any sufficiently small $\varepsilon_N \in (0,\vartheta_N)$, obtain
\begin{align}\label{eq:Q-bound-1}
            \big \| T_{\Z}\big[\llbracket 1 \mid \varphi_{\varepsilon_N}\rrbracket_{\mathcal Q(N)}\big] \delta_0 \big\|_{\ell^p(\Z)}^p
    = N^p \| \mathcal{F}_{\Z}^{-1}(\varphi_{\varepsilon_N}) \|_{\ell^p(N\Z)}^p
      \simeq N^{p-1} \varepsilon_N^{p-1},
\end{align}
using the $\ell^p$ scaling behavior of the kernel $\mathcal{F}_{\Z}^{-1}(\varphi_{\varepsilon_N})$. Applying \eqref{ce:eq1} with $G \equiv 1$ yields
\begin{equation*} 
\mathcal{F}_{\Z}^{-1} \big( \llbracket 1 \mid {\mathfrak m}_{\varepsilon_N}\rrbracket_{\mathcal R(N)}\big)(n)
= \mathcal{F}_{\Z}^{-1}(\mathfrak m_{\varepsilon_N})(n) h_n(N),
\end{equation*}
where $h_n(N)\coloneqq \sum_{u \in \mathcal R(N)} e(n \cdot u)$. Then, for any sufficiently small $\varepsilon_N \in (0,\vartheta_N)$, we obtain
\begin{align}\label{eq:R-bound-1}
\big \| T_{\Z}\big[\llbracket 1 \mid {\mathfrak m}_{\varepsilon_N}\rrbracket_{\mathcal R(N)}\big] \delta_0 \big\|_{\ell^p(\Z)}^p
= \sum_{n \in [N]} |h_n(N)|^p \| \mathcal{F}_{\Z}^{-1}(\mathfrak m_{\varepsilon_N}) \|_{\ell^p(N\Z + n)}^p
\simeq H(N) \varepsilon_N^{p-1},
\end{align}
where $H(N) \coloneqq |h_1(N)|^p / N + \dots + |h_N(N)|^p /N$. In the first identity we split the summation into the classes modulo $N$ and used the periodicity $h_{n+N}(N) = h_n(N)$. The second bound follows from the $\ell^p$ scaling behavior of the kernel $\mathcal{F}_{\Z}^{-1}({\mathfrak m}_{\varepsilon_N})$. 

\smallskip \paragraph{\bf Step~4} Given $n \in \Z$, we notice that $h_n(N_1 N_2) = h_n(N_1) h_n(N_2)$ holds for any $N_1,N_2 \in \Z_+$ such that ${\rm gcd} (N_1,N_2) = 1$. In particular, for $N = q_1 \cdots q_K$, where $q_1, \dots, q_K$ are the first $K$ prime numbers, we have 
\[
|h_n(N)| = |h_n(q_1) \cdots h_n(q_K)| = |(q_1 \, \ind{q_1 | n}  - 1) \cdots (q_K \, \ind{q_K | n}  - 1)| =
N a_1(n) \cdots a_K(n)
\]
with $a_k(n) \coloneqq |\ind{q_k | n} - 1/q_k|$ for each $k \in [K]$. We shall interpret the functions $a^p_1, \dots, a^p_K$ as independent random variables on $[N]$ with normalized counting measure, such that $a^p_k = (1-1/q_k)^p$ with probability $1/q_k$ and $a^p_k = (1/q_k)^p$ with probability $1-1/q_k$. We have
\[
H(N) = N^{p} \mathbb{E}(a_1^p \cdots a_K^p) = N^{p} \mathbb{E}(a_1^p) \cdots \mathbb{E}(a_K^p).
\]
Since $\mathbb{E}(a_k^p) = (1/q_k)(1-1/q_k)^p + (1-1/q_k)(1/q_k)^p$, we obtain the following formula
\[
\mathbb{E}(a_1^p) \cdots \mathbb{E}(a_K^p)
= \sum_{L \subseteq [K]} \Big( \prod_{l \in L} (1/q_l) (1 - 1/q_l)^p \prod_{l' \in [K] \setminus L} (1-1/q_{l'}) (1/q_{l'})^p \Big).  
\]
Denote $S_K \coloneqq (1-1/q_1) \cdots (1-1/q_K)$ and estimate the last quantity from below by
\[
S_K^p 
\sum_{L \subseteq [K]} \Big( \prod_{l \in L} (1/q_l) \prod_{l' \in [K] \setminus L} (1/q_{l'})^p \Big)
= S_K^p \prod_{k \in [K]} (1/q_k + 1/q_k^p)
= S_K^p N^{-1} \prod_{k \in [K]} (1 + q_k^{1-p}). 
\]

\smallskip \paragraph{\bf Step~5} Combining these estimates with \eqref{eq:lower bound 1}, \eqref{eq:Q-bound-1}, and \eqref{eq:R-bound-1}, we obtain the lower bound
\[
\big \|  T_{\Z}\big[\llbracket 1 \mid {\mathfrak m}_{\varepsilon_N}\rrbracket_{\mathcal R(N)}\big]\big\|_{\ell^p(\Z)\to\ell^p(\Z)}
\gtrsim
\prod_{k \in [K]} (1 - q_k^{-1})^p (1 + q_k^{1-p})
\]
for $N = q_1 \cdots q_K$. The natural logarithm of the right-hand side is equal to
\[
\sum_{k \in [K]} \big( \log ( 1 + q_k^{1-p}) + p \log(1-q_k^{-1}) \big).
\]
If $K$ is large, then the prime number theorem implies $q_K \simeq \log N$. Moreover, we have
\[
\sum_{k \in [K]} \log ( 1 + q_k^{1-p}) \gtrsim
\sum_{k \in [K]} q_k^{1-p}
\gtrsim q_K^{2-p} / \log q_K 
\]
by $\pi(q_K) - \pi(q_K/2) \gtrsim q_K / \log q_K$, where $\pi$ is the prime-counting function, and 
\[
\sum_{k \in [K]} - p \log(1-q_k^{-1}) \lesssim 
\sum_{k \in [K]} 1 / q_k \lesssim \log \log q_K
\]
by the second Mertens theorem. Fix $\delta \in \R_+$ such that $p+\delta\in(1, 2)$. Combining the above bounds, we arrive at
\[
\log \big \|  T_{\Z}\big[\llbracket 1 \mid {\mathfrak m}_{\varepsilon_N}\rrbracket_{\mathcal R(N)}\big]\big\|_{\ell^p(\Z)\to\ell^p(\Z)}
\gtrsim_\delta q_K^{2-p-\delta} \simeq (\log N)^{2-p-\delta}. 
\]
Thus, for infinitely many $N$, we have \eqref{eq:N-c_p} for all sufficiently small $\varepsilon_N \in (0, \vartheta_N)$.  
\end{proof}

\begin{remark}
\label{rem:0}
We close this section by emphasizing that the method of proving inequality \eqref{eq:N-c_p} (especially Steps~1--3) enables us to construct bounded weights $G$ satisfying condition~\ref{cond:A} and Schwartz functions ${\mathfrak m}$ supported on ${\bf Q}$ for which the following holds. If $p \in (1,\infty) \setminus \{2\}$, then there exists $\delta \in \R_+$ such that, for any sequence $(\vartheta_N)_{N \in \Z_+}\subseteq (0,1)$ and for infinitely many integers $N\in\Z_+$, we have
\begin{equation*}
     \big \|  T_{\Z}\big[\llbracket G \mid {\mathfrak m}_{\varepsilon_N}\rrbracket_{\mathcal R(N)}\big]f\big\|_{\ell^{p}(\Z)}
     \gtrsim_{p,\delta}
     N^\delta \|f\|_{\ell^p(\Z)}
\end{equation*}
with some $0 < \varepsilon_N \leq \vartheta_N$ and ${\mathfrak m}_{\varepsilon_N}(x) \coloneqq {\mathfrak m}(x/\varepsilon_N)$. Taking $r \in \NN_{\geq 2}$ and $\vartheta_N= 2^{-N}$, we see that \eqref{eq:support} holds for all sufficiently large $N\in\Z_+$, yet the conclusion of Theorem~\ref{thm:IW:upper} fails. Necessarily, this construction shows that the weights $G$ constructed above fail to satisfy condition \ref{cond:B}; otherwise, we would reach a contradiction with Theorem~\ref{thm:IW:upper}.
\end{remark}

\section{Seminorm variant of the Ionescu--Wainger theorem} \label{sec:var}

In this section, we extend the Ionescu--Wainger multiplier theorem to the seminorm setting. The main objects will be the so-called \textit{Rademacher--Menshov seminorms}, whose properties will be axiomatized below.

\begin{definition}[Rademacher--Menshov seminorm]\label{def:RM}

Fix $p \in [1,\infty)$ and $d \in \Z_+$. Let
\[
\mathcal{S}( f_n : n\in\mathbb{I} ) = \mathcal{S}_{L^p(X;H)} ( f_n : n\in\mathbb{I} )
\]
be a given quasiseminorm defined on all sequences $(f_n : n\in\mathbb{I})$ of functions $f_n \in L^p(X;H)$, where $H$ is any  Hilbert space and $X$ is either $\Z^d$ with counting measure or $\RR^d$ with Lebesgue measure, and $\mathbb{I}$ is any subset of $\NN$. Given a pair $(H_1,H_2)$ of  Hilbert spaces, we say that $\mathcal S$ is a \emph{Rademacher--Menshov seminorm} of type $(d,p,H_1,H_2)$ if the following conditions hold.
\begin{enumerate} [label*={\rm{(\arabic*)}}, itemsep=2pt]
\item \label{conditionA} Splitting property: there exists a constant $\mathbf C_{1} \in \R_+$ such that
\begin{equation*}
\mathcal{S}
( f_n : n\in\mathbb{N})
\le \mathbf C_{1}
\big(
\mathcal{S} ( f_n : n\in\mathbb{N}_{\leq K} )
+
\mathcal{S} ( f_n : n\in\mathbb{N}_{\geq K} )
\big)
\end{equation*}
holds for all $K \in \NN$ and all functions $f_n \in \ell^p(\Z^d;H_2)$.
\item \label{conditionB} Rademacher--Menshov inequality: there exists a constant $\mathbf C_{2} \in \R_+$ such that
\begin{equation*}
\mathcal{S}
( f_n : n \in \mathbb{N}_{\leq 2^J} )
\le \mathbf C_{2}
\sum_{j \in \NN_{\leq J}} 
\Big\|
\Big(
\sum_{i \in \NN_{<2^{J-j}}}
\big\| f_{2^{j}(i+1)} - f_{2^{j}i} \big\|_{H_2}^{2}
\Big)^{1/2}
\Big\|_{\ell^p(\Z^d)}
\end{equation*}
holds for all $J\in\mathbb{N}$ and all functions $f_n \in \ell^p(\Z^d;H_2)$.    
\item \label{conditionC} Sampling principle: there exists a constant $\mathbf C_{3} \in \R_+$ such that
\begin{equation*}
\big\|
\big(T_{\mathbb{Z}^d}\big[\llbracket 1 \mid \mathfrak{m}_n\rrbracket_{\mathcal{Q}(q)}\big]\big)_{n\in\mathbb{N}}
\big\|_{
  \ell^{p}(\mathbb{Z}^d;H_1)
  \to
  \mathcal{S}_{\ell^{p}(\mathbb{Z}^d;H_2)}
}
\le
\mathbf C_{3}
\big\|
\big(T_{\mathbb{R}^d}[\mathfrak{m}_n]\big)_{n\in\mathbb{N}}
\big\|_{
  L^{p}(\mathbb{R}^d;H_1)
  \to
  \mathcal{S}_{L^{p}(\mathbb{R}^d;H_2)}}
\end{equation*}
holds for all $q\in\mathbb{Z}_+$ and all multipliers $\mathfrak{m}_n \colon \mathbb{R}^d \to L(H_1,H_2)$ supported on ${q}^{-1}{\bf Q}^d$.
\end{enumerate}
\end{definition}

\subsection{Seminorm variant of the Ionescu--Wainger multiplier  theorem}
The following result, which is the main result of this section, is a refinement of \cite[Theorem~3.32]{KMPWW}.

\begin{theorem} \label{thm:IWvar}
Fix $r, d \in \mathbb{Z}_+$ and $N \in \mathbb{N}_{\geq 10}$. Let $(\lambda_n)_{n \in \mathbb{N}}\subseteq[1, \infty)$ be a lacunary sequence such that $\lambda \coloneqq \inf_{n\in\mathbb{N}} \lambda_{n+1}/\lambda_n > 1$. Let $H_1,H_2$ be separable Hilbert spaces. For each $n\in\mathbb{N}$, let $\mathfrak m_n \colon \mathbb{R}^{d} \to L(H_1,H_2)$ be a measurable function supported on $\varepsilon_n {\bf Q}^d$, where
\begin{align} \label{conditionC-epsilon}
0 < \varepsilon_n \leq \min\{(4r N^{2r})^{-1}, \lambda_n^{-1}\}. 
\end{align}
Let $G \colon \mathbb{Q}^d \to \mathbb{C}$ be a Ionescu--Wainger weight from Definition~\ref{def:IW-weight} and $\mathcal{S}$ be a~Rademacher--Menshov seminorm of type $(d, p,H_1,H_2)$ from Definition~\ref{def:RM} with $(2r)/(2r-1) \leq p \leq 2r$.  Let
\[
C_N^r(G) \coloneqq N^{{\bf C}_{\rm IW}(r) / \log \log N}\,
\mathbf{C}_{G}(\mathcal R^d_{\leq N})
\]
be the corresponding constant from Theorem~\ref{thm:IW:upper}. Then
\begin{align*}
\big \|
\big(T_{\mathbb{Z}^d}\big[\llbracket G \mid \mathfrak{m}_n\rrbracket_{\mathcal R^d_{\le N}}\big] \big)_{n\in\mathbb{N}}&
\big \|_{
  \ell^{p}(\mathbb{Z}^d;H_1)
  \to
  \mathcal{S}_{\ell^{p}(\mathbb{Z}^d;H_2)}
}
\lesssim_{\lambda,d,r} C_N^r(G)  \, 
\mathbf C_{1}
( \mathbf C_{2} {\bf A}_{2r} \log N 
+ \mathbf C_{3}
{\bf B}_{p})
\end{align*}
holds with 
\begin{align*}
    {\bf A}_{2r} & \coloneqq \sup_{\omega \in \{-1,1\}^\mathbb{N}}
	\Big\|\sum_{n\in\mathbb{N}} \omega(n)\,
 T_{\mathbb{R}^d}[\mathfrak m_{n+1}-\mathfrak m_n]\Big\|_{L^{2r}(\mathbb{R}^d;H_1)\to L^{2r}(\mathbb{R}^d;H_2)}, \\
    {\bf B}_{p} & \coloneqq 
    \big\|\big(T_{\mathbb{R}^d}[\mathfrak m_n]\big)_{n\in\mathbb{N}}\big\|_{L^{p}(\mathbb{R}^d;H_1)\to \mathcal{S}_{L^{p}(\mathbb{R}^d;H_2)}}.
\end{align*}
\end{theorem}

\begin{proof}[Proof of Theorem~\ref{thm:IWvar}]
Fix $f\in \ell^p(\mathbb{Z}^d;H_1)$ and set $F_n \coloneqq T_{\mathbb{Z}^d}\big[\llbracket G \mid \mathfrak{m}_n\rrbracket_{\mathcal R^d_{\le N}}\big] f$ for every $n \in \N$. Set $\kappa_N \coloneqq \lfloor \log_2 (10r (\log\lambda)^{-1}N )\rfloor + 1$. By condition~\ref{conditionA} with $K=2^{\kappa_N}$, we have
\[
\mathcal{S}_{\ell^p(\mathbb{Z}^d;H_2)}\big(F_n : n\in\mathbb{N}\big)
\le \mathbf C_{1}
\big(
\mathcal{S}_{\ell^p(\mathbb{Z}^d;H_2)} (F_n : n\in\mathbb{N}_{\leq 2^{\kappa_N}} )
+
\mathcal{S}_{\ell^p(\mathbb{Z}^d;H_2)} (F_n : n\in\mathbb{N}_{\geq 2^{\kappa_N}})
\big).
\]
Hence, it suffices to prove
\begin{align}
\label{eq:62}
\mathcal{S}_{\ell^p(\mathbb{Z}^d;H_2)} (F_n : n\in\mathbb{N}_{\leq 2^{\kappa_N}} )
&\lesssim_{\lambda, d,r}
C_N^r(G) \, \mathbf C_{2} \,
{\bf A}_{2r}\, (\log N) \, \|f\|_{\ell^p(\mathbb{Z}^d;H_1)},\\
\label{eq:63}
\mathcal{S}_{\ell^p(\mathbb{Z}^d;H_2)} (F_n : n\in\mathbb{N}_{\ge 2^{\kappa_N}} )
&\lesssim_{\lambda, d,r}
C_N^r(G) \, \mathbf C_{3} \,
{\bf B}_{p}\,
\|f\|_{\ell^p(\mathbb{Z}^d;H_1)}.
\end{align}
Theorem~\ref{thm:IWvar} follows by combining \eqref{eq:62} and \eqref{eq:63} with the splitting inequality above.

\smallskip \paragraph{\bf Step~1} 
We now prove inequality \eqref{eq:62}. By condition~\ref{conditionB} with $J={\kappa_N}$, we have
\begin{equation*}
\mathcal{S}_{L^p(\Z^d;H_2)}
( F_n : n \in \mathbb{N}_{\leq 2^{\kappa_N}} )
\le \mathbf C_{2}
\sum_{j \in \NN_{\leq {\kappa_N}}} 
\Big\|
\Big(
\sum_{i \in \NN_{<2^{{\kappa_N}-j}}}
\big\| F_{2^{j}(i+1)} - F_{2^{j}i} \big\|_{H_2}^{2}
\Big)^{1/2}
\Big\|_{\ell^p(\Z^d)}.
\end{equation*}
For each $j \in \NN_{\leq {\kappa_N}}$, by Khinchine's inequality, the corresponding summand is controlled by
\[
\bigg(
\int_0^1
\Big\|
\sum_{i \in \NN_{<2^{{\kappa_N}-j}}}
\omega_i(t)
(F_{2^j(i+1)}
 - F_{2^ji})
\Big\|_{\ell^p(\mathbb{Z}^d;H_2)}^p \, {\rm d}t
\bigg)^{1/p},
\]
where $(\omega_i)_{i \in \mathbb{N}}$ is a sequence of Rademacher functions. This leads to the bound 
\[ 
\sum_{j \in \NN_{\leq \kappa_N} } \bigg(\int_0^1 \big(
C_N^r(G)  \, {\bf A}_{2r} \, \|f\|_{L^p(\Z^d;H_1)} \big)^{p} \, {\rm d}t \bigg)^{1/p}, 
\] 
thanks to Theorem~\ref{thm:IW:upper}. Then \eqref{eq:62} follows, since $\kappa_N + 1 \simeq_{\lambda, r} \log N$.

\smallskip\paragraph{\bf Step~2} 
We now prove inequality \eqref{eq:63}. Let $\eta \colon \mathbb{R}\to\mathbb{R}$ be smooth, even, and such that $\ind{{\bf Q}}\le\eta\le \ind{2{\bf Q}}$. For each $m\in\mathbb{Z}$ and $\xi\in\R$ we set
\[
\eta_{\le m}(\xi) \coloneqq \eta(2^{-m}\xi).
\]
Let $L_N \in \mathbb{Z}_+$ be the least common multiple of the set $[N]$. If $n\ge 2^{\kappa_N}$, then
\[
\lambda_n\ge \lambda^n = e^{n \log \lambda} \ge e^{10rN} \geq 2^{5rN+5}\ge 10 L_N,
\]
by the well-known fact that $L_N \le 3^N$. Thus, we obtain the factorization
\begin{align*}
\llbracket G \mid \mathfrak{m}_n\rrbracket_{\mathcal R^d_{\le N}} 
=\llbracket 1 \mid \mathfrak{m}_n\rrbracket_{\mathcal Q(L_N)} 
\cdot \llbracket G \mid \eta_{\le -5rN}\rrbracket_{\mathcal R^d_{\le N}}.
\end{align*}
Indeed, $\lambda_n\ge 2^{5rN+5}$ implies that $\eta_{\le -5rN}$ equals $1$ on the support of $\mathfrak m_n$, while $2^{5rN+5}\ge 10 L_N$ implies that the pieces corresponding to different fractions do not overlap. The second factor defines an operator on $\ell^{p}(\mathbb{Z}^d;H_1)$ whose norm is controlled by $C_N^r(G)$, thanks to Theorem~\ref{thm:IW:upper}. For the first factor, by \eqref{conditionC-epsilon}, we can use condition~\ref{conditionC}. This gives 
\begin{align*} 
\big\| \big(T_{\mathbb{Z}^d} \big[ \llbracket G \mid \mathfrak{m}_n\rrbracket_{\mathcal R^d_{\le N}} \big] \big)_{n\in\mathbb{N}_{\geq 2^{\kappa_N}}} \big\|_{\ell^p(\mathbb{Z}^d;H_1)\to \mathcal{S}_{\ell^p(\mathbb{Z}^d;H_2)}} 
\lesssim_{d,r} \mathbf C_{3} \, {\bf B}_{p}
\cdot C_N^r(G). 
\end{align*}
Thus, \eqref{eq:63} holds and the proof is complete.
\end{proof}

\subsection{Examples of Rademacher--Menshov seminorms}
Fix a Hilbert space $H$ and parameters $\rho \in [1,\infty]$ and $\lambda \in \R_+$. For $J \in \Z_+$ and $\mathbb I \subseteq \NN$, let $\mathfrak S_J^{\mathbb I}$ be the space of all increasing sequences $({\mathfrak o}_j)_{j \in \NN_{\leq J}}$ of numbers ${\mathfrak o}_j \in {\mathbb I}$. Let $(h_n : n \in \NN)$ be a sequence of elements of $H$. 

Given $({\mathfrak o}_j)_{j \in \NN_{\leq J}} \in \mathfrak S_J^{\mathbb I}$, the \emph{$\rho$-oscillation seminorm} of $(h_n : n\in \mathbb I)$ is defined by
\begin{align*}
O^{\rho}_{H; ({\mathfrak o}_j)_{j \in \NN_{\leq J}}}(h_n : n\in \mathbb I)
\coloneqq 
\begin{cases}
\Big(\displaystyle\sum_{j \in \NN_{<J}}  
\displaystyle\sup_{n\in[{\mathfrak o}_j, {\mathfrak o}_{j+1}) \, \cap \, \mathbb I}
\|h_n-h_{{\mathfrak o}_{j}} \|_{H}^{\rho} \Big)^{1/\rho},  
& \rho \in \R_+,\\[1.5ex]
\displaystyle \sup_{j \in \NN_{<J}} \, \sup_{n\in[{\mathfrak o}_j, {\mathfrak o}_{j+1}) \, \cap \, \mathbb I}
\|h_n-h_{{\mathfrak o}_{j}} \|_{H},
& \rho = \infty.
\end{cases}
\end{align*}
Similarly, the \emph{$\rho$-variation seminorm} of $(h_n : n\in \mathbb I)$ is defined by
\begin{align*}
V_H^{\rho}(h_n : n\in\mathbb I)
\coloneqq 
\begin{cases}
\displaystyle\sup_{J\in\mathbb{N}}
\, \sup_{({\mathfrak n}_j)_{j \in \NN_{\leq J}} \in \mathfrak S_J^{\mathbb I}}
\Big(\displaystyle \sum_{j \in \NN_{< J}}  \|h_{{\mathfrak n}_{j+1}}-h_{{\mathfrak n}_{j}}\|_{H}^{\rho} \Big)^{1/\rho},  
&
\rho \in \R_+, \\[1.5ex]
\displaystyle \sup_{({\mathfrak n}_j)_{j \in \NN_{\leq 1}} \in \mathfrak S_1^{\mathbb I}}
\|h_{{\mathfrak n}_{1}}-h_{{\mathfrak n}_{0}}\|_{H},  
& \rho = \infty.
\end{cases}
\end{align*}
Finally, the \emph{$\lambda$-jump counting function} of $(h_n : n\in \mathbb I)$ is defined by
\begin{align*} 
N_{H; \lambda} (h_n : n\in \mathbb I) \coloneqq \sup \big\{J\in \Z_+ : 
\min_{j \in \NN_{< J}} \|h_{{\mathfrak n}_{j+1}}-h_{{\mathfrak n}_{j}}\|_{H} \geq \lambda \text{ for some } ({\mathfrak n}_j)_{j \in \NN_{\leq J}} \in \mathfrak S_J^{\mathbb I} \big\}.
\end{align*}
Here we use the convention that the supremum over the empty set equals $0$. If $H=\C$, then we abbreviate $O^{\rho}_{H; ({\mathfrak o}_j)_{j \in \NN_{\leq J}}}, V_{H}^{\rho}$, and $N_{H; \lambda}$ to $O^{\rho}_{({\mathfrak o}_j)_{j \in \NN_{\leq J}}}, V^{\rho}$, and $N_{\lambda}$, respectively. 

The above objects give rise to quasiseminorms on $L^p$ spaces. Let $(X, \mathcal{B}, \mathfrak m)$ be a $\sigma$-finite measure space and fix $p \in \R_+$. Let $(f_n : n \in \NN)$ be a sequence of functions $f_n \in L^p(X;H)$.

The \emph{$\rho$-oscillation seminorm} of $(f_n : n \in \mathbb I)$ is defined by
\begin{align}
\label{eq:osc-Lp}
  \begin{split}
\mathcal{O}_{H; X}^{\rho, p} (f_n : n \in \mathbb I) 
\coloneqq \sup_{J\in\mathbb{Z}_+} \, \sup_{({\mathfrak o}_j)_{j \in \NN_{\leq J}} \in \mathfrak S_J^{\mathbb I}}
\big\| O^{\rho}_{H; ({\mathfrak o}_j)_{j \in \NN_{\leq J}}}(f_n : n\in \mathbb I)\big\|_{L^{p}(X)}.
  \end{split}
\end{align}
Similarly, the \emph{$\rho$-variation seminorm} of $(f_n : n \in \mathbb I)$ is defined by
\begin{align}
\label{eq:var-Lp}
  \begin{split}
\mathcal{V}_{H; X}^{\rho, p}(f_n : n\in \mathbb I) 
\coloneqq \big\| V_H^{\rho}(f_n : n\in \mathbb I)\big\|_{L^{p}(X)}.
  \end{split}
\end{align}
Finally, the \emph{$\rho$-jump quasiseminorm} of $(f_n : n \in \mathbb I)$ is defined by
\begin{align}
 \label{eq:jump-Lp}
  \begin{split}
  \mathcal{J}_{H; X}^{\rho, p}(f_n : n\in \mathbb I) \coloneqq \sup_{\lambda \in \RR_+}
  \big\| \lambda \, N_{H; \lambda}(f_n : n\in \mathbb I)^{1/\rho} \big\|_{L^{p}(X)}.
  \end{split}
\end{align}
As before, we use the convention that the supremum over the empty set equals $0$. If $H=\C$, then we abbreviate $\mathcal{O}_{H; X}^{\rho, p}, \mathcal{V}_{H; X}^{\rho, p}$, and $ \mathcal{J}_{H; X}^{\rho, p}$ to $\mathcal{O}_{X}^{\rho, p}, \mathcal{V}_{X}^{\rho, p}$, and $ \mathcal{J}_{X}^{\rho, p}$, respectively. 
\begin{remark}
\label{rem:1}
The first two objects \eqref{eq:osc-Lp} and \eqref{eq:var-Lp} define seminorms (i.e. they are homogeneous and subadditive), while the third one \eqref{eq:jump-Lp}, as it stands, defines only a quasiseminorm (i.e. it is homogeneous and only satisfies the triangle inequality with some constant bigger than one). However, this is not an issue and, as shown in \cite[Corollary~2.2, p.~805]{MSZ1}, there exists a (subadditive) seminorm equivalent to the jump quasiseminorms from \eqref{eq:jump-Lp}.
\end{remark}

We now formulate the following result, which has important practical significance. 

\begin{proposition}
Fix $p \in [1,\infty)$ and $d \in \Z_+$. Let $H_1, H_2$ be  separable Hilbert spaces. If $\rho \in [2,\infty]$, then each of \eqref{eq:osc-Lp}--\eqref{eq:jump-Lp} is a~Rademacher--Menshov seminorm of type $(d, p,H_1,H_2)$. 
\end{proposition}

\begin{proof}
It is clear that each of \eqref{eq:osc-Lp}--\eqref{eq:jump-Lp} satisfies condition~\ref{conditionA}. Moreover, by $\rho \in[2,\infty]$, each of \eqref{eq:osc-Lp}--\eqref{eq:jump-Lp} is dominated by $\mathcal{V}_{H_2; X}^{2,p}$. Thus, condition~\ref{conditionB} follows from the Rademacher--Menshov inequality for $\mathcal{V}_{H_2; X}^{2,p}$, see \cite[Lemma~2.5, p.~534]{MSZ2}. For \eqref{eq:osc-Lp} and \eqref{eq:var-Lp}, condition~\ref{conditionC} holds by Theorem~\ref{prop:msw} after passing to suitable quotient spaces on which these seminorms become genuine norms and invoking a limiting argument to approximate $H_1, H_2$ by finite-dimensional subspaces. For \eqref{eq:jump-Lp}, condition~\ref{conditionC} was verified in \cite[Theorem~1.3, p.~800]{MSZ1}.
\end{proof}

\begin{remark}
\label{rem:2}
We close this section by emphasizing that each object from \eqref{eq:osc-Lp}--\eqref{eq:jump-Lp} has a certain splitting property along long and short times. Namely, for a sequence $(f_n : n \in \NN)$ of functions $f_n \in L^p(X;H)$ with $p \in [1,\infty)$, if $\mathcal{S}_{L^p(X;H)}^{\rho} ( f_n : n\in\mathbb{I} )$ with $\rho\in[2, \infty]$ is one of the seminorms from \eqref{eq:osc-Lp}--\eqref{eq:jump-Lp}, then we can bound it from above by a constant multiple of
\begin{align}
\label{eq:5}
\mathcal{S}_{L^p(X;H)}^{\rho} ( f_{2^n} : n\in\mathbb{N} )
+
\bigg\|\Big(\sum_{k\in\N}V_H^{2}\big(f_n : n\in\mathbb I\cap[2^k, 2^{k+1}]\big)^2\Big)^{1/2}\bigg\|_{L^p(X)}.
\end{align}
Applying such a bound is referred to as the so-called splitting into long and short jumps, see \cite[Lemma~1.3, p.~6716]{JSW} for details. While in practical problems the first term in \eqref{eq:5} can be handled by Theorem~\ref{thm:IWvar}, the other term can be handled by square function techniques. Namely, using \cite[Lemma~2.5, p.~534]{MSZ2}, we bound it from above by a constant multiple of
\begin{align}
\label{eq:6}
\sum_{l\in\N}\bigg\|\Big(\sum_{k\in\N}\sum_{m\in\N_{<2^l}}\|f_{2^k+2^{k-l}(m+1)}-f_{2^k+2^{k-l}m}\|_H^2\Big)^{1/2}\bigg\|_{L^p(X)}.
\end{align}
Now it is clear that \eqref{eq:6} falls into a square function estimate paradigm and to control the corresponding norm, we will employ Theorem~\ref{thm:IW:upper} as long as a  gain summable in $l\in\N$ is expected. In the vast majority of interesting situations, this is the case.
\end{remark}

\section{Bourgain's ergodic theorem}
\label{sec:B}
In the final section, we demonstrate how Theorem~\ref{thm:IW:upper} and Theorem~\ref{thm:IWvar} can be used in practice. Specifically, we will give a short proof of Bourgain's pointwise ergodic theorem with polynomial iterates. 

\begin{theorem}[Bourgain's ergodic theorem]
\label{bourgain}
Let $(X,\mathcal B(X), \mu)$ be a $\sigma$-finite measure space equipped with an invertible measure-preserving transformation $T \colon X\to X$, and $P\in \mathbb Z[{\rm n}]$ be a polynomial with integer coefficients. Then, for every $p\in(1, \infty)$ and every $f\in L^p(X)$, the polynomial ergodic averages
\begin{align}
\label{eq:7}
A_{N;X, T}^{P}f(x) \coloneqq \frac{1}{N}\sum_{n\in[N]}f(T^{P(n)}x), \qquad x\in X,
\end{align}
where $N\in\Z_+$, converge $\mu$-almost everywhere on $X$ and in $L^p(X)$ norm as $N\to\infty$. If, in addition, $\mu(X)=1$ and $T$ is totally ergodic (i.e. $T^n$ is ergodic for every $n\in \Z_+$), then
\begin{align}
\label{eq:12}
\lim_{N\to \infty}A_{N;X, T}^{P}f(x)=\int_Xf(t)\,{\rm d}\mu(t)
\end{align}
holds $\mu$-almost everywhere on $X$ and in $L^p(X)$ norm.
\end{theorem}
Theorem~\ref{bourgain} was proved by Bourgain in a series of groundbreaking papers \cite{B1, B2, B3}, giving an affirmative answer to the so-called Bellow--Furstenberg problem, see \cite{Bel} and \cite{F}. In fact, mean convergence for \eqref{eq:12} was established by Furstenberg \cite{Fur2} long before Bourgain's papers \cite{B1, B2, B3}. Theorem~\ref{bourgain} is a far-reaching generalization of Birkhoff's pointwise ergodic theorem \cite{BI} and von Neumann's mean ergodic theorem \cite{vN}. However, in contrast to Birkhoff's result, almost everywhere convergence of \eqref{eq:7} may fail for integrable functions $f\in L^1(X)$ when $\deg P \ge 2$, as shown by Buczolich and Mauldin in \cite{BM} for $P(n)=n^2$ and by LaVictoire in \cite{LaV1} for $P(n)=n^k$ with any integer $k\ge 2$. More about this theorem and the ergodic context can be found in the second author's survey \cite{M-ICM} as well as in \cite{RW}. 

Here we prove a quantitative version of Theorem~\ref{bourgain}, which reads as follows.

\begin{theorem}
\label{bourgain1}
Let $(X,\mathcal B(X), \mu)$ be a $\sigma$-finite measure space equipped with an invertible measure-preserving transformation $T \colon X\to X$, and $P\in \mathbb Z[{\rm n}]$ be a polynomial with integer coefficients. Define $\mathbb D_{\tau} \coloneqq \{\tau^n:n\in\N\}$ for any $\tau\in(1, \infty)$. Let $A_{N;X, T}^{P}$ be the ergodic average defined in \eqref{eq:7} with $N\in\Z_+$. Then the following statements hold.

\begin{itemize}
\item[\rm{(i)}] ($\rho$-variation ergodic theorem) For every $p\in(1, \infty)$ and $\rho\in(2, \infty)$ and $\tau\in(1, \infty)$, there exists $C_{p, \rho, \tau, P}\in\RR_+$ such that, for every $f\in L^p(X)$, we have 
\begin{align}
\label{eq:4}
\mathcal{V}_{X}^{\rho, p}(A_{N;X, T}^{P}f: N\in \mathbb D_{\tau}) \le C_{p, \rho, \tau, P}\|f\|_{L^p(X)}.
\end{align}
\item[\rm{(ii)}] ($2$-jump ergodic theorem) For every $p\in(1, \infty)$ and
$\tau\in(1, \infty)$, there exists $C_{p, \tau, P}\in\RR_+$ such that, for every $f\in L^p(X)$, we have 
\begin{align}
\label{eq:8}
  \mathcal{J}_{X}^{2, p}(A_{N;X, T}^{P}f: N\in \mathbb D_{\tau}) \le C_{p, \tau, P}\|f\|_{L^p(X)}.
\end{align}

\item[\rm{(iii)}] ($2$-oscillation ergodic theorem) For every $p\in(1, \infty)$ and $\tau\in(1, \infty)$, there exists $C_{p, \tau, P}\in\RR_+$ such that, for every $f\in L^p(X)$, we have 
\begin{align}
\label{eq:9}
\mathcal{O}_{X}^{2, p} (A_{N;X, T}^{P}f: N\in \mathbb D_{\tau}) \le C_{p,  \tau, P}\|f\|_{L^p(X)}.
\end{align}
\item[\rm{(iv)}] (Square function estimates) For every $p\in(1, \infty)$, there exist $c_p\in(0, 1)$ and $C_{p, P}\in\RR_+$ such that, for every $l\in\N$ and for every $f\in L^p(X)$, we have
\begin{align}
\label{eq:10}
\bigg\|\Big(\sum_{k\in\N}\sum_{m\in\N_{<2^l}}|\Delta_m A_{2^k+2^{k-l}m;X, T}^{P}f|^2\Big)^{1/2}\bigg\|_{L^p(X)}\le C_{p, P}2^{-c_pl}\|f\|_{L^p(X)},
\end{align}
where $\Delta_m A_{2^k+2^{k-l}m;X, T}^{P}\coloneqq A_{2^k+2^{k-l}(m+1);X, T}^{P}-A_{2^k+2^{k-l}m;X, T}^{P}$ for $m\in\N$.
\end{itemize}
\end{theorem}

Now a few comments about Theorem~\ref{bourgain} and Theorem~\ref{bourgain1} are in order.

\begin{enumerate}[label*={(\arabic*)}, itemsep=2pt]
\item There are two ways to obtain pointwise convergence for $A_{N;X, T}^{P}f$ from \eqref{eq:7}.

    \begin{itemize}[itemsep=2pt]

    \item[(a)] Each of the inequalities from \eqref{eq:4}--\eqref{eq:9} implies Theorem~\ref{bourgain}. We first obtain a $\mu$-almost everywhere limit along the set $\mathbb D_{\tau}$. Since $\tau \in (1,\infty)$ is arbitrary, using the positivity of $A_{N;X, T}^{P}$ (i.e. $A_{N;X, T}^{P}f\ge 0$ if $f\ge 0$) we obtain that the limit exists $\mu$-almost everywhere on $X$ as $N\to\infty$, see \cite[Lemma~1.5]{RW} for details.

    \item[(b)] The other approach is more flexible and applies to operators that are not necessarily positive. Namely, if we combine inequalities from \eqref{eq:4}--\eqref{eq:9} for $\tau=2$ with square function inequality \eqref{eq:10} and invoke the upper bounds \eqref{eq:5} and \eqref{eq:6}, then the inequalities from \eqref{eq:4}--\eqref{eq:9} hold with $\Z_+$ in place of $\mathbb D_{\tau}$ and so the implied constants in \eqref{eq:4}--\eqref{eq:9} can be taken independent of $\tau$. Now, if any of the inequalities from \eqref{eq:4}--\eqref{eq:9} holds with $\Z_+$ in place of $\mathbb D_{\tau}$, we immediately see that the limit of $A_{N;X, T}^{P}f$ exists $\mu$-almost everywhere on $X$ as $N\to\infty$.

    \end{itemize}

\item Each of the inequalities from \eqref{eq:4}--\eqref{eq:9} with $\tau=2$ implies that for every $p\in(1, \infty)$ there exists $C_{p,  P}\in\RR_+$ such that, for every $f\in L^p(X)$, we have
\begin{align}
\label{eq:11}
\big\|\sup_{N\in\Z_+}|A_{N;X, T}^{P}f|\big\|_{L^p(X)} \le C_{p,  P}\|f\|_{L^p(X)}.
\end{align}
Maximal inequality \eqref{eq:11}, together with the pointwise convergence of $A_{N;X, T}^{P} f$ and the dominated convergence theorem, implies convergence of $A_{N;X, T}^{P} f$ in $L^p(X)$ norm.

\item In view of the previous two items, Theorem~\ref{bourgain1} implies Theorem~\ref{bourgain}. The ranges of parameters in Theorem~\ref{bourgain1} are sharp. Inequality \eqref{eq:4} was established in \cite{MST2}, inequality \eqref{eq:8} was established in \cite{MSZ3}, whereas inequality \eqref{eq:9} was established in \cite{MSS}. Variants of the square function estimates in the spirit of \eqref{eq:10} were studied in \cite{MST2}. In fact, the authors in \cite{MSS, MST2, MSZ3} studied multidimensional variants of \eqref{eq:7} as well as its singular integral variants. Moreover, all inequalities from \eqref{eq:4}--\eqref{eq:9} were established along $\mathbb{Z}_+$ in place of $\mathbb D_{\tau}$ and so the implied constants were shown to be independent of $\tau$. Furthermore, they were also shown to be independent of the coefficients of the underlying polynomials.

\item Finally, we remark that inequality \eqref{eq:8} implies \eqref{eq:4} for all $\rho \in (2, \infty)$. This was observed by Bourgain in \cite{B3}, see also \cite{JSW, MSZ1}. From this point of view, inequality \eqref{eq:8} can be thought of as an endpoint estimate for \eqref{eq:4}. Interestingly, the bound 
\[
O^{2}_{({\mathfrak o}_j)_{j \in \NN_{\leq J}}}(A_{N;X, T}^{P}f : N\in \mathbb I)\lesssim V^{2}(A_{N;X, T}^{P}f : N\in \mathbb I)
\]
holds pointwise for any $({\mathfrak o}_j)_{j \in \NN_{\leq J}}$ and $J\in\Z_+$, and for any $\mathbb I\subseteq \Z_+$, while we have
\[
\sup_{\|f\|_{L^p(X)}\le1}\mathcal{V}_{X}^{2, p}(A_{N;X, T}^{P}f: N\in \mathbb D_{2})=\infty. 
\]
The authors in \cite{MSS} showed that there are no obvious relations between oscillation inequality \eqref{eq:9} and variation inequality \eqref{eq:4} beyond the trivial pointwise bound. The latter suggests that oscillation inequality \eqref{eq:9} should play a role of an endpoint for \eqref{eq:4} similar to jump inequality \eqref{eq:8}. Therefore, a natural question arises whether the bounds from \eqref{eq:9} can be used to deduce estimates from \eqref{eq:4} for all $\rho \in (2, \infty)$.
\end{enumerate}

We now focus on proving Theorem~\ref{bourgain1}. In fact, we only prove inequalities \eqref{eq:4} and \eqref{eq:10}. Inequalities \eqref{eq:8} and \eqref{eq:9} can be derived in the same way as inequality \eqref{eq:4}.

\subsection{Preliminary reductions}
First, we reduce the matter to the integer shift system $(\mathbb Z, \mathcal B(\mathbb Z), \mu_{\mathbb Z})$ endowed with the shift transformation $S \colon \mathbb Z\to\mathbb Z$ defined by $S(x) \coloneqq x-1$ for  $x\in\mathbb Z$. Here $\mathcal B(\mathbb Z)$ is the $\sigma$-algebra of all subsets of $\mathbb Z$ and $\mu_{\mathbb Z}$ is counting measure on $\mathbb Z$. In this case, for every $N \in \Z_+$, we shall abbreviate $A_{N;\mathbb Z, S}^{P}f$ to
\begin{align}
\label{eq:13}
A_{N;\mathbb Z}^{P}f(x) \coloneqq \frac{1}{N}\sum_{n\in [N]}f(x-P(n)), \qquad x\in \mathbb Z.
\end{align}
In view of the Calder{\'o}n transference principle \cite{Cald}, see also \cite{Kosz}, it suffices to prove Theorem~\ref{bourgain1} for the integer shift system. Then each of the inequalities \eqref{eq:4}--\eqref{eq:10} for the integer shift system immediately implies its analog for an arbitrary measure-preserving system.

We also remark that the Calder{\'o}n transference principle \cite{Cald, Kosz} only transfers quantitative bounds (such as inequalities \eqref{eq:4}--\eqref{eq:10}) that imply pointwise almost everywhere convergence, but does not transfer pointwise almost everywhere convergence from the integer shift system to an arbitrary measure-preserving system. In fact, in the integer shift system, pointwise convergence is implied by norm convergence, since the $\ell^\infty(\mathbb Z)$ norm is dominated by the $\ell^2(\mathbb Z)$ norm. This is not necessarily true for other measure-preserving systems.

The advantage of working  with the integer shift system is that the average $A_{N;\mathbb Z}^{P}f$ can be expressed as the convolution of $f$ with the kernel
\begin{align*}
K_N^P\coloneqq \frac{1}{N}\sum_{n\in[N]}\ind{\{P(n)\}}.
\end{align*}
The averages from \eqref{eq:13} are discrete averaging operators of Radon type, and their integral counterparts are well understood through classical harmonic analysis methods, see \cite{bigs} and also \cite{DF} or \cite{MSZ2}. Using the Fourier transform, we have 
\begin{align*}
\mathcal F_{\mathbb Z}(A_{N; \mathbb Z}^{P} f)(\xi)=m_{N}(\xi)\mathcal F_{\mathbb Z}f(\xi), \qquad \xi\in\mathbb T,
\end{align*}
where the multiplier $m_N$ is a normalized exponential sum of the form 
\begin{align*}
m_{N}(\xi) \coloneqq \mathcal F_{\mathbb Z}K_N^P(\xi)= \frac{1}{N}\sum_{n\in[N]}e(-\xi\cdot P(n)).
\end{align*}

Here is the place where the classical circle method enters the game and can be used to understand the nature of the multiplier \(m_{N}\) by analyzing the major and minor arcs. This will require the concepts of canonical fractions and their corresponding major arcs.

Let $\eta \colon \mathbb R\to[0, 1]$ be a smooth and even function satisfying
\begin{align*}
\ind{[-1/4, 1/4]}\le \eta\le \ind{(-1/2, 1/2)}.
\end{align*}
For any $n, \xi\in\mathbb R$, we set $\eta_{\le n}(\xi) \coloneqq \eta(2^{-n}\xi)$. For any $l\in\mathbb N$, using the definition of canonical fractions \eqref{IWeq:372}, we set $\Sigma_{\leq l} \coloneqq \mathcal R_{\le 2^l}$ and $\Sigma_l  \coloneqq  \Sigma_{\leq l} \setminus \Sigma_{\leq l-1}$. Note that $|\Sigma_{\leq l}| \le 2^{2l}$. Similarly, for any $l\in\mathbb N$ and $m\in\mathbb Z$, we introduce the dyadic \emph{major arcs} by
\begin{align*}
{\mathcal M}_{\leq l, \leq m}  \coloneqq \bigcup_{\theta\in \Sigma_{\leq l}}[\theta-2^m, \theta+2^m].
\end{align*}
We note that the sets ${\mathcal M}_{\le l, \leq m}$ are nondecreasing in both $l$ and $m$, and if $m \le - 2l - 2$, then the arcs $[\theta - 2^m, \theta + 2^m]$ that comprise ${\mathcal M}_{\leq l, \leq m}$ are pairwise disjoint. We also define
\begin{align*}
{\mathcal M}_{l, \leq m}   \coloneqq   {\mathcal M}_{\leq l, \leq m} \setminus {\mathcal M}_{\leq l-1, \leq m}
\quad \text{and} \quad
{\mathcal M}_{l,m}  \coloneqq   {\mathcal M}_{l,\leq m} \setminus {\mathcal M}_{l,\leq m-1}.
\end{align*}

Using \eqref{eq:1}, we define the dyadic Ionescu--Wainger projections $\Pi_{\leq l, \leq m} \colon \ell^2(\mathbb Z) \to \ell^2(\mathbb Z)$ by 
\begin{align*}
\Pi_{\leq l, \leq m} f \coloneqq T_{\mathbb{Z}}\big[\llbracket 1 \mid \eta_{\le m}\rrbracket_{\Sigma_{\le l}}\big] f ,\qquad f \in \ell^2(\mathbb Z).
\end{align*}
If $m \le - 2l - 2$, then the  following properties of the Ionescu--Wainger projections hold.

\begin{enumerate}[label*={(\roman*)}, itemsep=2pt]
\item The operator $\Pi_{\leq l, \leq m}$ is self-adjoint and real symmetric.
\item The function $\mathcal F_{\mathbb Z} (\Pi_{\leq l, \leq m} f)$ is supported on the set ${\mathcal M}_{\leq l, \leq m}$.
\item If $\mathcal F_{\mathbb Z} f$ vanishes on the set ${\mathcal M}_{\leq l, \leq m}$, then $\Pi_{\leq l, \leq m} f \equiv 0$.
\item If $\mathcal F_{\mathbb Z} f$ is supported on the set ${\mathcal M}_{\leq l, \leq m-2}$, then $\Pi_{\leq l, \leq m} f = f$.
\item  The operator $\Pi_{\leq l, \leq m}$ is a contraction on $\ell^2(\mathbb Z)$.
\item Moreover, for any $p\in(1, \infty)$ and $\gamma\in(0, 1)$, there exists a constant $C_p\in\mathbb R_+$ such that, for every $l\in\mathbb N$ and $m \in \mathbb Z$, if $m \leq - 6 \max\{p, p/(p-1)\}(l+1)$, then
\begin{align}
\label{eq:116}
\| \Pi_{\leq l, \leq m} f \|_{\ell^p(\mathbb Z)} \lesssim_{p,\gamma} 2^{C_p \gamma l}\|f\|_{\ell^p(\mathbb Z)}.
\end{align}
\end{enumerate}
Inequality \eqref{eq:116} is a consequence of Theorem~\ref{thm:IW:upper}.

\subsection{Minor arcs estimates}
We fix $p \in (1,\infty)$ and a polynomial $P\in\mathbb Z[{\rm n}]$ with $\deg P=d\ge 2$ . Then we choose $p_0\in 2\mathbb Z_+$ such that $p_0/(p_0-1)<p<p_0$. One may think that $p_0$ is a sufficiently large even integer, which serves as the parameter $p$ from inequality \eqref{eq:116} and the parameter $2r$ from Theorem~\ref{thm:IW:upper} and Theorem~\ref{thm:IWvar}. It ensures that inequality \eqref{eq:116} as well as Theorem~\ref{thm:IW:upper} and Theorem~\ref{thm:IWvar} can be applied. If necessary, $p_0\in 2\mathbb Z_+$ may be further adjusted to the interpolation arguments that will be used throughout this section.

We fix a small parameter 
	\begin{align*}
	0<\alpha<(10^6 dp_0)^{-1}
	\end{align*}
and a large integer $C_0\in \Z_+$ depending on $p_0$ and $P$. For all $N\in \Z_+$, we define the quantities 
	\begin{align*}
	l_{(N)} \coloneqq {\rm Log\,} N^\alpha \quad \text{and} \quad L_{(N)} \coloneqq {\rm Log\,} N-l_{(N)},
	\end{align*}
using the log-scale notation ${\rm Log\,} N \coloneqq \lfloor \log_2 N\rfloor$. For each $N \geq C_0$, we now decompose $A_{N; \mathbb{Z}}^P f$ into minor-arc and major-arc pieces, respectively, as follows
\[
A_{N; \mathbb Z}^Pf=A_{N; \mathbb Z}^P (f -\Pi_{\leq l_{(N)}, \leq -d L_{(N)}} f)
+A_{N; \mathbb Z}^P(\Pi_{\leq l_{(N)}, \leq -d L_{(N)}}f).
\]

By inequality \eqref{eq:116}, for any $\gamma\in(0, 1)$, we have
\begin{align}
\label{eq:18}
\big\|A_{N; \mathbb Z}^P(f -\Pi_{\leq l_{(N)}, \leq -d L_{(N)}} f) \big\|_{\ell^{p_0}(\mathbb Z)} 
\lesssim_{\gamma} N^{\gamma} \|f\|_{\ell^{p_0}(\mathbb Z)}.
\end{align}
This inequality can be improved by using  Weyl's inequality for the multiplier $m_N$.

\begin{proposition}
For every $C\in\mathbb R_+$, there exists a small constant $c\in(0, 1)$ such that, for all $N\in\Z_+$ and $\delta\in (0, 1]$, if $\xi \in \TT \setminus \bigcup_{\theta\in \mathcal R_{\le \delta^{-C}}}[\theta-N^{-d}\delta^{-C}, \theta+N^{-d}\delta^{-C}]$, then
\begin{align}
\label{eq:19}
|m_{N}(\xi)|\le c^{-1}(\delta^c+N^{-c}).
\end{align}
\end{proposition}
Inequality \eqref{eq:19} is the classical Weyl sum estimate for normalized exponential sums, which can be deduced from its inverse form, see for instance \cite[Exercise~1.1.21, p.~16]{Tho}.

Since $\mathcal F_{\Z}(f -\Pi_{\leq l_{(N)}, \leq -d L_{(N)}} f)$ vanishes on the set ${\mathcal M}_{\leq l_{(N)}, \leq -d L_{(N)}}$, we can use Plancherel's theorem and inequality \eqref{eq:19} to conclude that there exists a constant $c_2\in(0, 1)$ such that
\begin{align}
\label{eq:20}
\big\|A_{N; \mathbb Z}^P(f -\Pi_{\leq l_{(N)}, \leq -d L_{(N)}} f) \big\|_{\ell^{2}(\mathbb Z)} 
\lesssim N^{-c_2} \|f\|_{\ell^{2}(\mathbb Z)}.
\end{align}

By interpolating \eqref{eq:18} with \eqref{eq:20} and taking $\gamma\in(0, 1)$ sufficiently small, we deduce that there exists a small absolute constant $c_p\in(0, 1)$ such that, for all $f\in \ell^p(\mathbb Z)$, we have 
\begin{align}
\label{eq:21}
\big\|A_{N; \mathbb Z}^P(f -\Pi_{\leq l_{(N)}, \leq -d L_{(N)}} f) \big\|_{\ell^p(\mathbb Z)} 
\lesssim N^{-c_p} \|f\|_{\ell^p(\mathbb Z)}.
\end{align}
Let $\mathbb D_{\tau,*} \coloneqq \{ N \in \mathbb D_{\tau} : N \geq C_0 \}$. Using \eqref{eq:21}, we obtain
\begin{align}
\label{eq:22}
\mathcal{V}_{\Z}^{\rho, p}\big(A_{N; \mathbb Z}^P(f -\Pi_{\leq l_{(N)}, \leq -d L_{(N)}} f) : N\in \mathbb D_{\tau,*}\big) \lesssim \sum_{N\in \mathbb D_{\tau,*}}N^{-c_p} \|f\|_{\ell^p(\mathbb Z)}\lesssim_{\tau}\|f\|_{\ell^p(\mathbb Z)},
\end{align}
since $\sum_{N\in \mathbb D_{\tau,*}}N^{-c_p}\lesssim_{\tau} 1$. Now it remains to understand the major-arc piece.
\subsection{Major arcs estimates}
By the splitting property and \eqref{eq:22}, it remains to show 
\begin{align}
\label{eq:23}
\mathcal{V}_{\Z}^{\rho, p} \big(A_{N; \mathbb Z}^P(\Pi_{l, \leq -d L_{(N)}}f)
:N\in\mathbb D_{\tau, l} \big)
\lesssim 2^{-c_pl}\|f\|_{\ell^p(\mathbb Z)}
\end{align}
for all $l\in\mathbb N$ with
$\mathbb D_{\tau, l} \coloneqq \{N\in\mathbb D_{\tau}:N\ge C_0 \text{ and
} N\ge 2^{l/\alpha}\}$. Indeed, since
\begin{align*}
A_{N; \mathbb Z}^P(\Pi_{\leq l_{(N)}, \leq -d L_{(N)}}f)=\sum_{l \in \N_{\le l_{(N)}}}
A_{N; \mathbb Z}^P(\Pi_{l, \leq -dL_{(N)}}f),
\end{align*}
summing \eqref{eq:23} over $l\in\mathbb N$, and using the splitting property and \eqref{eq:22}, we obtain \eqref{eq:4}.

For each $N\ge 1$, we define a continuous counterpart of $m_N$ by
\begin{align*}
\mathfrak{m}_N(\xi) \coloneqq \int_{0}^1e(-\xi\cdot P(Nt)) \,{\rm d}t, \qquad \xi\in\mathbb R.
\end{align*}
For every $(a,q) \in\mathbb Z \times \mathbb Z_+$ such that ${\rm gcd}(a, q)=1$, we define the complete exponential sum $G(a/q) \coloneqq m_q(a/q)$. A simple application of the mean value theorem (see \cite[Lemma~7.20]{KMPWW}) shows that, for every $\xi\in\mathbb T$ and $\theta\in \Sigma_{l}$ such that $|\xi-\theta|\le M^{-1}$ for some $M \geq 1$, one has
\begin{align}
\label{eq:25}
|m_N(\xi)-G(\theta)\mathfrak m_N(\xi-\theta)|\lesssim_P 2^{l} (M^{-1}N^{d-1}+N^{-1} ).
\end{align}

By inequality \eqref{eq:25} with $M=2^{dL_{(N)}}$, there exists a constant $c_p\in(0, 1)$ such that
\[
A_{N; \mathbb Z}^P(\Pi_{l, \leq -dL_{(N)}}f)=T_{\mathbb Z}\big[\llbracket G \mid \mathfrak m_N\eta_{\le -d L_{(N)}}\rrbracket_{\Sigma_{l}}\big]f+E_Nf,
\]
where $\|E_Nf\|_{\ell^p(\mathbb Z)} \lesssim N^{-c_p}\|f\|_{\ell^p(\mathbb Z)}$. Arguing as in \eqref{eq:22}, and using the last inequality for $N\in\mathbb D_{\tau, l}$, we reduce the verification of inequality \eqref{eq:23} to proving that
\begin{align}
\label{eq:26}
\mathcal{V}_{\Z}^{\rho, p}\big(T_{\mathbb Z}\big[\llbracket G \mid \mathfrak m_N\eta_{\le -d L_{(N)}}\rrbracket_{\Sigma_{l}}\big]f:N\in\mathbb D_{\tau, l}\big)
\lesssim 2^{-c_pl}\|f\|_{\ell^p(\mathbb Z)}.
\end{align}
By Theorem~\ref{thm:IWvar}, for every $\gamma\in(0, 1)$, we have
\begin{align}
\label{eq:27}
\mathcal{V}_{\Z}^{\rho, p_0}\big(T_{\mathbb Z}\big[\llbracket G \mid \mathfrak m_N\eta_{\le -d L_{(N)}}\rrbracket_{\Sigma_{l}}\big]f:N\in\mathbb D_{\tau, l}\big)
\lesssim_{\gamma} 2^{\gamma l}\|f\|_{\ell^{p_0}(\mathbb Z)},
\end{align}
since, for every $q\in(1, \infty)$, we have 
\begin{align*}
{\bf A}_{q} & \coloneqq \sup_{\omega \in \{-1,1\}^\mathbb{N}}
\Big\|\sum_{n\in\mathbb{N}} \omega(n)\,
T_{\mathbb{R}^d}[\mathfrak m_{\lfloor \tau^{n+1}\rfloor}\eta_{\le -d L_{(\lfloor \tau^{n+1}\rfloor)}}-\mathfrak m_{\lfloor \tau^{n}\rfloor}\eta_{\le -d L_{(\lfloor \tau^{n}\rfloor)}}]\Big\|_{L^{q}(\mathbb{R})\to L^{q}(\mathbb{R})}<\infty, \\
{\bf B}_{q} & \coloneqq \sup_{\|f\|_{L^q(\R)}\le1}\mathcal{V}_{\R}^{\rho, q}\big(T_{\mathbb R}\big[\mathfrak m_{\lfloor \tau^{n}\rfloor}\eta_{\le -d L_{(\lfloor \tau^{n}\rfloor)}}\big]f:n\in\N\big)<\infty.
\end{align*}
The bounds for ${\bf A}_{q}$ and  ${\bf B}_{q}$ follow respectively from \cite{DF} and  \cite{JSW}. We also refer to \cite{MSZ2} for a unified approach, which also allows us to handle jump and oscillation seminorms.

Inequality \eqref{eq:27} can be improved by controlling the  complete exponential sum $G$ on $\Sigma_{l}$. Indeed, define $u\coloneqq 100(l+1)$ and note that $N\ge \max\{2^{l/\alpha}, C_0\}$ for any $N\in\mathbb D_{\tau, l}$. This implies $N\ge 2^{10^5  dp_0l}C_0^{1/2}\ge 2^{100dp_0u}$, provided that $C_0\ge 2^{10^6  dp_0}$. Thus, we may write
\begin{align}
\label{eq:28}
T_{\mathbb Z}\big[\llbracket G \mid \mathfrak m_N\eta_{\le -d L_{(N)}}\rrbracket_{\Sigma_{l}}\big]f
=
T_{\mathbb Z}\big[\llbracket 1 \mid \mathfrak m_N\eta_{\le -d L_{(N)}}\rrbracket_{\Sigma_{\le l}}\big]
\, T_{\mathbb Z}\big[\llbracket G \mid \eta_{\le -d up_0}\rrbracket_{\Sigma_{l}}\big]f.
\end{align}
By Plancherel's theorem and the Weyl sum estimate for complete exponential sums, we immediately deduce that there exists a constant $c\in(0, 1)$ such that
\begin{align}
\label{eq:29}
\big\|T_{\mathbb Z}\big[\llbracket G \mid \eta_{\le -dp_0u }\rrbracket_{\Sigma_{l}}\big]f\big\|_{\ell^2(\Z)}\lesssim 2^{-cl}\|f\|_{\ell^2(\Z)}.
\end{align}
Combining \eqref{eq:27} for $p_0=2$ and $\gamma=c/2$ with \eqref{eq:29} and the identity \eqref{eq:28}, we obtain 
\begin{align}
\label{eq:30}
\mathcal{V}_{\Z}^{\rho, 2}\big(T_{\mathbb Z}\big[\llbracket G \mid \mathfrak m_N\eta_{\le -d L_{(N)}}\rrbracket_{\Sigma_{l}}\big]f:N\in\mathbb D_{\tau, l}\big)
\lesssim 2^{-c_2 l}\|f\|_{\ell^{2}(\mathbb Z)}
\end{align}
with $c_2 = c/2$. Interpolating \eqref{eq:27} with \eqref{eq:30}, and taking $\gamma\in(0, 1)$ to be sufficiently small, we obtain inequality \eqref{eq:26} for some $c_p\in(0,  1)$. This  completes the proof of inequality \eqref{eq:4}.

\subsection{Concluding remarks} As we remarked above, the proofs of inequalities \eqref{eq:8} and \eqref{eq:9} proceed with obvious changes, in much the same way as the proof of inequality \eqref{eq:4}. We omit the details. We now briefly outline the proof of the square function estimate from \eqref{eq:10}, which in the integer shift settings reads as follows
\begin{align}
\label{eq:31}
\bigg\|\Big(\sum_{k\in\N}\sum_{m\in\N_{<2^l}}|\Delta_m A_{2^k+2^{k-l}m;\Z}^{P}f|^2\Big)^{1/2}\bigg\|_{\ell^p(\Z)}\le C_{p, P}2^{-c_pl}\|f\|_{\ell^p(\Z)},
\end{align}
where $\Delta_m A_{2^k+2^{k-l}m;\Z}^{P}\coloneqq A_{2^k+2^{k-l}(m+1);\Z}^{P}-A_{2^k+2^{k-l}m;\Z}^{P}$ for $m\in\N$.

Since the left-hand side of \eqref{eq:31} is defined by a square function, we can entirely rely on Theorem~\ref{thm:IW:upper}. However, the challenge is to gain decay in $l \in \mathbb{N}$. For this purpose, the key observation explaining how the decay in  $l \in \mathbb{N}$ arises is that, for all \(p \in [1, \infty)\), we have
\begin{align}
\label{eq:16}
\big\|A_{2^k+2^{k-l}(m+1);\Z}^{P}f-A_{2^k+2^{k-l}m;\Z}^{P}f\big\|_{\ell^p(\Z)}\lesssim 2^{-l}\|f\|_{\ell^p(\Z)},
\end{align}
since $2^k+2^{k-l}m\simeq 2^k$ for all $m\in \N_{<2^l}$. To prove \eqref{eq:31}, we will again analyze the minor and major arcs separately. On minor arcs, for every $k\in\N$ and $\gamma\in(0, 1)$, we have
\begin{align}
\label{eq:32}
\bigg\|\Big(\sum_{m\in\N_{<2^l}}\big|\Delta_m A_{2^k+2^{k-l}m;\Z}^{P} (f-\Pi_{\leq l_{(2^k)}, \leq -d L_{(2^k)}}f ) \big|^2\Big)^{1/2}\bigg\|_{\ell^{p_0}(\Z)}\lesssim_\gamma 2^{\gamma k}\|f\|_{\ell^{p_0}(\Z)},
\end{align}
which follows from \eqref{eq:16} and inequality \eqref{eq:116}. Combining \eqref{eq:16} with \eqref{eq:20}, we obtain
\begin{align}
\label{eq:33}
\big\|\Delta_m A_{2^k+2^{k-l}m;\Z}^{P}(f-\Pi_{\leq l_{(2^k)}, \leq -d L_{(2^k)}}f ) \big\|_{\ell^{2}(\Z)}
\lesssim \min\{2^{-l}, 2^{-c_2k}\}\|f\|_{\ell^{2}(\Z)}.
\end{align}
Using \eqref{eq:33} and bounding the minimum by $2^{-3l/4 - c_2k/4}$, we see that
\begin{align}
\label{eq:34}
\bigg\|\Big(\sum_{m\in\N_{<2^l}}\big|\Delta_m A_{2^k+2^{k-l}m;\Z}^{P}(f-\Pi_{\leq l_{(2^k)}, \leq -d L_{(2^k)}}f )\big|^2\Big)^{1/2}\bigg\|_{\ell^{2}(\Z)}\lesssim 2^{-l/4- c_2k/4}\|f\|_{\ell^{2}(\Z)}.
\end{align}
Interpolating \eqref{eq:34} with \eqref{eq:32} and taking $\gamma\in(0, 1)$ sufficiently small, and summing over $k\in\NN$, we obtain \eqref{eq:31} with $f-\Pi_{\leq l_{(2^k)}, \leq -d L_{(2^k)}}f$ in place of $f$. Now the proof of inequality \eqref{eq:31} is reduced to showing \eqref{eq:31} with $\Pi_{\leq l_{(2^k)}, \leq -d L_{(2^k)}}f$ in place of $f$. For this purpose, we can mimic the argument from the major-arc estimates, as in the proof of inequality \eqref{eq:4}. After major arcs estimates based on \eqref{eq:25} with $M=2^{dL_{(2^k)}}$, followed by an application of Theorem~\ref{thm:IW:upper}, we are reduced to showing that there exists a constant $c_p\in(0, 1)$ such that
\begin{align}
\label{eq:35}
\bigg\|\Big(\sum_{k\in\N}\sum_{m\in\N_{<2^l}}|T_{\R}\big[(\Delta_m \mathfrak m_{2^k+2^{k-l}m;\Z})\eta_{\le -d L_{(2^k)}}\big]f|^2\Big)^{1/2}\bigg\|_{L^p(\R)}\lesssim 2^{-c_pl}\|f\|_{L^p(\R)}.
\end{align}
Inequality \eqref{eq:35} follows from square function considerations combined with the estimate $\|T_{\R}[(\Delta_m \mathfrak m_{2^k+2^{k-l}m;\Z})]f\|_{L^p(\R)}\lesssim 2^{-l}\|f\|_{L^p(\R)}$, valid for all $p\in[1, \infty)$, see \cite{MSZ2} for details. 

This completes the proof of inequality \eqref{eq:31} and  the proof of inequality \eqref{eq:10} follows.

\end{document}